\documentclass[11pt,oneside,english,reqno]{amsart}
\usepackage[T1]{fontenc}
\usepackage[utf8]{inputenc}
\usepackage[a4paper]{geometry}
\usepackage{mathrsfs}
\usepackage{amstext}
\usepackage{amsthm}
\usepackage{amssymb}
\usepackage{color}
\usepackage{hyperref}
\usepackage{enumerate}
\usepackage{verbatim} 
\usepackage{stmaryrd}
\usepackage{enumitem}
\usepackage{graphicx}
\usepackage{todonotes}

\usepackage{caption}
\usepackage{subcaption}

\makeatletter
\numberwithin{equation}{section}
\numberwithin{figure}{section}
\theoremstyle{plain}
\newtheorem{thm}{\protect\theoremname}
\theoremstyle{definition}
\newtheorem{defn}[thm]{\protect\definitionname}
\theoremstyle{remark}
\newtheorem{rem}[thm]{\protect\remarkname}
\theoremstyle{plain}
\newtheorem{lem}[thm]{\protect\lemmaname}
\theoremstyle{plain}

\theoremstyle{plain}
\newtheorem{cor}[thm]{\protect\corollaryname}
\theoremstyle{plain}
\newtheorem{example}[thm]{\protect\examplename}
\theoremstyle{plain}

\makeatother

\usepackage{babel}
\providecommand{\corollaryname}{Corollary}
\providecommand{\definitionname}{Definition}
\providecommand{\lemmaname}{Lemma}
\providecommand{\propositionname}{Proposition}
\providecommand{\remarkname}{Remark}
\providecommand{\theoremname}{Theorem}
\providecommand{\examplename}{Example}
\providecommand{\conjecturename}{Conjecture}

\newcommand{\cA}{\mathcal{A}}

\newcommand{\cF}{\mathcal{F}}
\newcommand{\cG}{\mathcal{G}}
\newcommand{\cH}{\mathcal{H}}
\newcommand{\cI}{\mathcal{I}}

\newcommand{\cP}{\mathcal{P}}
\newcommand{\cQ}{\mathcal{Q}}

\newcommand{\EE}{\mathbb{E}}

\newcommand{\NN}{\mathbb{N}}

\newcommand{\PP}{\mathbb{P}}

\newcommand{\RR}{\mathbb{R}}

\newcommand{\bk}{\mathbf{k}}

\newcommand{\bl}{\mathbf l}

\newcommand{\bm}{\mathbf{m}}

\newcommand{\bn}{\mathbf{n}}

\newcommand{\bp}{\mathbf{p}}

\newcommand{\bR}{\mathbf{R}}
\newcommand{\br}{\mathbf{r}}
\newcommand{\bS}{\mathbf{S}}
\newcommand{\bT}{\mathbf{T}}

\newcommand{\bx}{\mathbf{x}}

\newcommand{\by}{\mathbf{y}}

\newcommand{\boo}{\mathbf{0}}

\newcommand{\balpha}{{\boldsymbol\alpha}}
\newcommand{\bepsilon}{{\boldsymbol\epsilon}}

\newcommand{\bu}{\mathbf{u}}
\newcommand{\bv}{\mathbf{v}}
\newcommand{\bs}{\mathbf{s}}
\newcommand{\bt}{\mathbf{t}}

\newcommand{\bz}{\mathbf{z}}

\newcommand{\Id}{\mathrm{Id}}

\newcommand{\dd}{\mathop{}\!\mathrm{d}}

\newcommand{\norm}[1]{\left\lVert #1\right\rVert}
\newcommand{\abs}[1]{\left\lvert #1\right\rvert}

\begin{document}

\title[]{A multiparameter Stochastic Sewing lemma and the regularity of local times associated to Gaussian sheets}
\author{Florian Bechtold \and Fabian A. Harang \and Hannes Kern}
\date{\today}

\address{Florian Bechtold: Fakult\"at f\"ur Mathematik,
Universit\"at Bielefeld, 33501 Bielefeld, Germany}
\email{fbechtold@math.uni-bielefeld.de}

\address{Fabian A. Harang: Department of Economics, BI Norwegian Business School, Handelshøyskolen BI, 0442, Oslo, Norway.}
\email{fabian.a.harang@bi.no}

\address{Hannes Kern: Departement II, TU Berlin, 10623 Berlin, Germany.}
\email{kern@math.tu-berlin.de}

\thanks{ 
\\
\emph{Acknowledgment:} FB acknowledges support through the Bielefeld Young Researchers Fund that enabled a research visit. HK acknowledges funding by DFG through IRTG 2544.}


\begin{abstract}
We establish a multiparameter extension of the stochastic sewing lemma \cite{stochasticsewing}. This allows us to derive novel regularity estimates on the local time of locally non-deterministic Gaussian fields. These estimates are sufficiently strong to derive regularization by noise results for SDEs in the plain by leveraging the results from \cite{Bechtold2023}. In this context, we make the interesting and rather surprising observation that regularization effects profiting from each parameter of the underlying stochastic field in an additive fashion usually appear to be due to boundary terms of the driving stochastic field.
\end{abstract}

\maketitle

\tableofcontents{}

\section{Introduction}
\noindent Setting up an integration theory with respect to Brownian motion $W$ or more general stochastic processes in order to study problems of the form
\begin{equation}
    \label{sde_intro}
    \dd X_t=b(X_t)\dd t +\sigma(X_t)\dd W_t, \qquad x_0\in \RR^d
\end{equation}
is usually an intricate task in the sense that the naive strategy of fixing typical realizations and performing a pathwise analysis fails. This requires us to exploit additional probabilistic properties of the underlying process, which in It\^o's theory of stochastic calculus consists in the martingale property. One alternative approach to stochastic calculus consists in rough path theory due to Lyons \cite{Lyons1998}, where probabilistic properties of Brownian motion allow the construction of the iterated integral $\mathbb{W}_{s,t}=\int_s^t (W_r-W_s) \otimes \dd W_r$. Once this object is constructed, one can show that the couple $(W, \mathbb{W})$ (also called a rough path) is enough to establish an entirely pathwise theory of stochastic integration and thus study associated stochastic differential equations (refer to \cite{frizhairer} for a standard reference on this subject). Crucial advantages of this "factorization" of the problem \eqref{sde_intro} into two steps (i.e. constructing the "lift" in step one and then going from an 'enhanced' local approximation back to the global objects in step two in order to study \eqref{sde_intro}) consist in 
\begin{itemize}
    \item the continuity of the map $(W,\mathbb W)\to X$ (also called It\^o-Lyons map) in suitable topologies, i.e. the mapping that returns the solution $X$ of \eqref{sde_intro} given some noise $X$ (refer to for example to \cite[Theorem 8.5]{frizhairer}).
    \item as a consequence, large deviation principles for \eqref{sde_intro} and related systems can be obtained rather directly thanks to the contraction principle (refer to for example \cite[Theorem 9.5]{frizhairer}). 
    \item possibility to address noise $W$ which is not a semi-martingale, provided a rough path lift is still available (example: fractional Brownian motion with Hurst parameter $H>1/4$. 
    \item possibility to infinite dimensional generalizations, which arguably reached their culmination point in the theory of regularity structures \cite{Hairer2014} and paracontrolled distributions \cite{gubinelli_imkeller_perkowski_2015}. 
\end{itemize}
While extremely versatile in theory and applications, there are however certain questions for which the rough path perspective is too coarse, precisely because it discards fine probabilistic arguments from a very early stage on. In particular, intricate properties such as stochastic cancellations as expressed by the famous Burkholder-Davies-Gundy inequality for example can not be captured by this approach as the later analysis is entirely pathwise.  

\bigskip

In this context, the seminal Stochastic Sewing Lemma due to Khoa L\^e \cite{Le20} can be seen as an approach that combines rough analysis viewpoints with fine probabilistic arguments. In particular, this tool allows us to carry local stochastic cancellations over to global objects, thereby opening up an entirely new line of research. Let us briefly sketch the main idea of its proof and its improvement with respect to the classical sewing lemma. Given some local approximation $A:\Delta_T\to \RR^d$, we ask ourselves under what conditions on $A$ the Riemann-type sums
\[
I^n_t=\sum_{[u, v]\in \mathcal{P}^n([0,t])}A_{u,v}
\]
converge to an object $I_t$, independent of the sequence $(\mathcal{P}^n([0,t]))_n$ of partitions chosen as $|\mathcal{P}^n([0, t])|\to 0$. The by now classical Sewing Lemma due to Gubinelli \cite{gubi} states that this is the case, provided that the following bound hold
\[
|\delta_{s,u,t}A|:=|A_{s,t}-A_{s,u}-A_{u, t}|
\lesssim |t-s|^{1+\epsilon}
\]
for all $s<u<t\in [0, T]$ and some $\epsilon>0$. Note that in the case of a stochastic process $A:\Omega \times \Delta_T \to \RR^d$ all these considerations can of course be carried out pathwise. However, this viewpoint also adapted by rough path theory might be blind to some additional local stochastic cancellations that could be exploited. Indeed, suppose that the stochastic process $A:\Omega \times \Delta_T \to \RR^d$ is $(\mathcal{F}_t)_t$ adapted, i.e. $A_{s,t}$ is $\mathcal{F}_t$ measurable and assume further $\mathbb{E}[\delta_{s,u,t}A|\mathcal{F}_s]=0$. In this case, note that for dyadic partitions $\mathcal{P}^n([0,t])$, we have 
\[
I^n_t-I^{n+1}_t=\sum_{k=0}^{2^n-1} \Big((\delta_{t_k, u_k, t_{k+1}}A)-\mathbb{E}[(\delta_{t_k, u_k, t_{k+1}}A)|\mathcal{F}_{t_k}]\Big), 
\]
where $u_k=\frac{t_k+t_{k+1}}{2}$.
Remarking that the above represents a sum over martingale differences, we obtain immediately from the Burkholder-Davies-Gundy and Minkowski's inequality that 
\[
\norm{I^n_t-I^{n+1}_t}_{L^m(\Omega)}\lesssim \left(\sum_{k=0}^{2^n-1}\norm{\delta_{t_k, u_k, t_{k+1}}A)}_{L^m(\Omega)}^2\right)^{1/2}
\]
from which we infer that $(I^n)_n$ is Cauchy in $L^m(\Omega)$ provided that 
\[
\norm{\delta_{s,u,t}A}_{L^m(\Omega)}\lesssim |t-s|^{1/2+\epsilon}
\]
for all $s<u<t\in [0, T]$ and some $\epsilon>0$. Remark that under the additional assumption $\mathbb{E}[A_{s,t}|\mathcal{F}_s]=0$, one is able to half the local regularity condition on the delta-operator (at the price of also working the $L^m(\Omega)$ topology instead of the topology of almost sure convergence). 

\bigskip

One particularly dynamic field, where stochastic sewing is crucially used consists in pathwise regularization by noise \cite{galeati2020noiseless, harang2020cinfinity, galeatiharang,particlesystems,catellierharang,catellier2,dareiotis,gerencsergal,lukasz,Galeati2022,galeati2,tolomeo, toyomu, butkovsky2023stochastic, bechtold}. While there are many results in this direction in the one parameter setting, i.e. for stochastic processes, treating stochastic fields and by extension regularization by noise for SPDEs with these techniques has so far been done less systematically. A first prominent result for SPDEs using the stochastic sewing Lemma consists in \cite{athreya2022wellposedness} where the authors are able to establish well-posedness of the stochastic heat equation with certain distributional drifts covering in particular the case of the skewed stochastic heat equation. Another line of research related in spirit is \cite{Bechtold2023}, where stochastic differential equations in the plane are treated. Similar to \cite{harang2020cinfinity}, the authors establish a 2D non-linear Young theory, provided certain regularity assumptions on the local time of the underlying stochastic field can be made (also refer to Section \ref{regularizatio by noise section} for a more detailed overview of this approach). While \cite{Bechtold2023} does provide some concrete examples, for which such regularity of the local time results can be provided, a systematic study of which conditions on stochastic fields imply regular local times is missing.

\bigskip

The goal of the present paper is two-fold: We first provide a multiparameter version of the stochastic sewing Lemma, applicable to stochastic fields and of interest in its own right. The analysis is inspired by \cite{kern2023stochastic} already providing a stochastic reconstruction theorem and \cite{Harang2021} establishing a multiparameter sewing lemma. In the second part, we show how this multiparameter stochastic sewing lemma can be used in order to derive new regularity estimates for the local time of stochastic fields. Similar to the one-parameter setting \cite{harang2020cinfinity}, we identify local non-determinism conditions as the crucial property that allows us to derive such results. In combination with the 2D nonlinear Young theory already established in \cite{Bechtold2023}, this provides us immediately with regularization by noise results for SDEs in the plane driven by locally non-deterministic stochastic fields. Such SDEs are tightly connected to the stochastic non-linear wave equation, as explained in more detail in \cite{Bechtold2023}.

\subsection*{Sketch of main ideas}
Let us briefly sketch some main ideas and concepts that go into our analysis. A major difference between stochastic processes (i.e. indexed by time, for example, $[0, T]$) and fields (i.e. indexed by multiparameters $[0, T_1]\times \dots \times [0,T_d]$) consists in the fact that there is no canonical ordering of the index variable. This lack of a canonical "past" immediately translates into some ambiguity as to how to define filtrations for stochastic fields, see for example Figure \ref{filtrations}. Indeed, different orderings on $[0, T_1]\times \dots \times [0, T_d]$ yield different filtrations with different inconveniences and advantages. In our setting, we made the choice of working with the filtration $(\mathcal{F}_\bt)_\bt$ induced by the ordering $\bx\leq \by $, where  $x_i\leq y_i$ for all $1, \dots, d$ , see Figure \ref{filtrations} (A), which we call the strong past. This choice is motivated by the possibility of later incorporating known results in the literature on strong and sectorial local non-determinism \cite{Xiao} into our framework, allowing us to use our multiparameter stochastic sewing lemma to establish new regularity estimates for the local times of such fields. \\
\\
Once our filtration generated by the strong past is fixed, we require some further structural properties from it to establish the multiparameter stochastic sewing lemma. As we intend to iterate the martingale difference / BDG argument along the parameter dimension, we will require a certain compatibility of conditioning with respect to "one"-parameter projections (see Figure \ref{filtrations} (C), (D)). Filtrations that satisfy this naturally appearing structural property are known in the literature as commuting filtrations (refer to Definition \ref{commuting def}).\\
\\
From an algebraic side, we also require the notion of rectangular increments and associated adapted $\delta$-operators, which canonically generalize the one parameter setting of the classical \cite{gubi} and stochastic sewing lemma \cite{stochasticsewing}. Towards this end, we adopt the notational conventions of \cite{Harang2021}. Once this probabilistic and algebraic framework is fixed, we are in shape to prove our multiparameter stochastic sewing lemma.
\\
\\
With the multiparameter stochastic sewing lemma at hand, our main application consists in providing novel regularity estimates for local times of stochastic fields, which are sufficiently strong to deduce pathwise regularization by noise results. The main example we have in mind consists of the fractional Brownian sheet (Definition \ref{fbs}). Local times of stochastic fields are a well-established object of study in the literature, we refer to \cite{Dozzi2003, horowitz} for excellent review articles. Typical regularity results one obtains in the literature concern higher order "spatial" regularity, i.e. the regularity of $x\to L_\bt(x)$ for fixed $\bt\in [0, \bT]$ based on Fourier techniques (see for example \cite[Theorem 28.1]{horowitz})  or the joint continuity of $(x, \bt)\to L_{\bt}(x)$.  Joint continuity in the case of the fractional Brownian sheet for example was established in \cite{Xiao2002}. Let us also point out regularity estimates in Sobolev-Watanabe spaces for the fractional Brownian sheet due to \cite{Tudor2003}. However, we require quantified regularity results for $(x, \bt)\to L_{\bt}(x)$ on a H\"older-Bessel scale almost surely in order to establish regularization by noise results in the spirit of \cite{Bechtold2023}. We are able to provide such regularity estimates in Theorem \ref{additive LND local time} thanks to a combination of Fourier techniques with the multiparameter stochastic sewing lemma in its simplified version, inspired by a corresponding one-parameter result in \cite{harang2020cinfinity}. Using this derived regularity on H\"older-Bessel scales, we are able to directly harness the machinery developed in \cite{Bechtold2023} to deduce pathwise regularization by noise for SDEs in the plain. \\
\\
The crucial structural property of stochastic fields that allows for such arguments are variants of local non-determinism (LND) which we call additive and multiplicative local non-determinism. Local non-determinism is also a classical area of study going back to the works of Berman \cite{berman} and Pitt \cite{pitt}. For a survey on local non-determinism and different variants thereof, we refer to \cite{Xiao2009}. We discuss the relation of our introduced notions of additive and multiplicative LND with other LND notions in the literature and provide explicit calculations for the fractional Brownian sheet verifying the multiplicative LND condition. In this context, we make the observation that additive LND is a property essentially due to boundary terms of the stochastic field considered. This suggests that additive regularization (i.e. increased spatial regularity of the associated local time profiting from each parameter individually in an additive fashion, see Theorem \ref{additive LND local time}) is an effect coming from boundary terms.

\subsection*{Organization of the paper}
After fixing the probabilistic and algebraic framework in Section \ref{preliminaries}, we establish the multiparameter stochastic sewing lemma in Section \ref{sewing section}. In Section \ref{LND section} we then briefly recall different notions of local non-determinism for Gaussian fields introduced in the literature and how they relate to a formulation of said notion that we employ next. In Section \ref{sec_Local_Time} we show that using local non-determinism of Gaussian fields and our derived multiparameter stochastic sewing lemma, one can establish regularity estimates for the local time of such fields on H\"older-Bessel scales. These regularity estimates turn out to be sufficient to immediately apply results from \cite{Bechtold2023}, allowing us to conclude regularization by noise phenomena for SDEs in the plain driven by such fields. 

\section{Preliminaries on multiparameter Stochastics}
\label{preliminaries}
\subsection{Rectangular increments}

Throughout this paper, we will call the dimension of our setting $d$. Let us introduce some frequently used notation:

\begin{itemize}
    \item $[d] = \{1,\dots, d\}$ denotes the full index set.

    \item For any indexset $\theta\in[d]$, its compliment is given by $\theta^c := [d]\setminus\theta$.
    
    \item Letters in light fonts denote numbers $s\in\RR$, $n\in\NN$, whereas letters in bold fonts denote points $\bs\in\RR^d$ or multiindices $\bn\in\NN^d$. In particular, the $d$-dimensional vector consisting of only number 1, is written $\mathbf{1}$, i.e. $\mathbf{1}=(1,...,1)$. Similarly, $\mathbf{0} := (0,\dots,0)$. 

    \item For a multiparameter sequence $(h_\bn)_{\bn\in\NN^d}$, we say that $h_\bn$ converges to $h$ as $\bn\to\infty$, if $h_{\bn^{(k)}}$ converges to $h$ as $k\to\infty$ for all sequences $(\bn^{(k)})_{k\in\NN}$ such that $\mbox{min}_{i\in[d]}(n^{(k)}_i)$ goes to $\infty$ as $k\to\infty$.
    
    \item We identify any indexset $\theta\subset[d]$ with the vector $(\theta_i)_{i=1,\dots,d}$, $\theta_i = 1$ for all $i\in\theta$ and $\theta_i = 0$ else. This especially implies that $[d] = \mathbf 1$ and we can write $\bx+\theta = \by$ for the vector $\by = (y_i)_{i\in[d]}$ such that $y_i = x_i+1$ for $i\in\theta$ and $y_i = x_i$ else.

    \item For two points $\bx,\by\in\RR^d$, we write $\bx\le\by$ if $x_i\le y_i$ holds for $i=1,\dots,d$. We use $\ge,<,>$ analogously.

    \item We denote by $T$ our time horizon. If we work in a multiparameter setting, our time horizon will be given by $[0,T_1]\times\dots\times[0,T_d] =: [\boo,\bT]$.

    \item $\Delta_\bT$ is set to be the simplex of $[\boo,\bT]$ given by

    \begin{equation*}
        \Delta_\bT := \{(\bs,\bt)\in[\boo,\bT]^2~\vert~\bs\le\bt\}\,.
    \end{equation*}
\end{itemize}

\noindent We introduce rectangular increments, mainly following \cite{Harang2021} and \cite{Bechtold2023}. As a motivation, consider a smooth function $f:[0,T]\to\RR$. The increment of $f$ between to numbers $s\le t$ is then given by $f(t)-f(s) = \int_s^t f'(r) dr$. This approach very neatly generalizes to rectangles in $[\mathbf 0,\bT]$: Given two points $\bs\le\bt$ in $[\mathbf 0,\bT]$, we will set the rectangular increment of $f$ between $\bs$ and $\bt$ to be
\begin{equation}\label{eq:motivatingIntegrals}
    \square_{\bs,\bt}^{[d]} f = \int_{s_1}^{t_1}\dots\int_{s_d}^{t_d} \frac{\partial^d}{\partial r_1\dots\partial r_d}f(\br)\dd \br\,.
\end{equation}
This expression can be written as a difference over $f$ evaluated at the edges of the hypercube spanned between $\bs,\bt$, which we will use to extend this definition to non-differentiable $f$. To do so, we need an efficient way to project points onto each other:

\begin{defn}
    For $i\in[d]$ and $\bs\in[\mathbf 0,\bT]$, we define $\pi^i_\bs :[\mathbf 0,\bT]\to[\mathbf 0,\bT]$ to be the projection of the $i$-th variable onto $\bs$:
    \begin{equation*}
        \pi^i_\bs:\bt\mapsto(t_1,\dots, t_{i-1},s_i,t_{i+1},\dots,t_d)\,.
    \end{equation*}
    For any indexset $\theta\subset[d]$, we set
    \begin{equation*}
        \pi_\bs^\theta := \prod_{i\in\theta}\pi_\bs^i:\bt\mapsto(x_i)_{i\in[d]}\,,
    \end{equation*}
    where $x_i = s_i$ for $i\in\theta$ and $x_i = t_i$ else. For a function $f:[\mathbf 0,\bT]\to\RR$, we define

    \begin{equation*}
        \pi_\bs^\theta f(\bt) := f(\pi_\bs^\theta\bt)\,.
    \end{equation*}
\end{defn}
\noindent With this definition, we can define the square increment in a rigorous way:
\begin{defn}
    Given a function $f:[0,T]^d\to\bR$ and an index set $\theta\subset\{1,...,d\}$, we define the rectangular increment to be
    \begin{equation}
        \square_{\bs,\bt}^\theta f := \prod_{i\in\theta}(\Id-\pi^i_\bs)f(\pi^{\theta^c}_\bs \bt)\,,
    \end{equation}
    for any $\bs\le \bt$.
\end{defn}

\noindent This notation is best known in $d=2$, where the increments read
\begin{align*}
    \square_{\bs,\bt}^{(1)} f &= f(t_1,s_2)-f(s_1,s_2)\\
    \square_{\bs,\bt}^{(2)} f &= f(s_1,t_2)-f(s_1,s_2)\\
    \square_{\bs,\bt}^{(1,2)} f &= f(t_1,t_2)-f(t_1,s_2)-f(s_1,t_2)+f(s_1,s_2)\,.
\end{align*} 
One quickly checks that for smooth functions and $\theta = [d]$, this notation agrees with \eqref{eq:motivatingIntegrals}. It is further not hard to see that the square increment fulfills the following identities:
\begin{align*}
    \square_{\bs,\bt}^\theta f &= \prod_{i\in\theta}(\pi_\bt^i-\Id)f(\bs)\\
    &=\sum_{\eta\subset\theta} (-1)^{\abs{\theta\setminus\eta}}f(\pi^\eta_\bt\bs)
\end{align*}
The delta operator $\delta$ has received much attention in the rough paths theory (e.g. \cite{frizhairer}) in connection with various forms of the sewing lemma. It describes how to cut apart increments in a suitable way. A multiparameter extension of this operator is necessary for our purposes, and we therefore recall the construction from \cite{HARANG202134}.

For pairs $(\bs,\bt)\in[\boo,\bT]\times[\boo,\bT]$ and $\bu\in[\boo,\bT]$, we write
\begin{equation}
    \psi^\eta_\bu (\bs,\bt) := (\pi^\eta_\bu\bs,\bt) + (\bs,\pi^\eta_\bu\bt)\,,
\end{equation}
and set $\psi_\bu^\eta f(\bs,\bt) = f(\psi_\bu^\eta(\bs,\bt))$, as before. We can then define:

\begin{defn}
    For any index $i\in[d]$ and $\Xi:\Delta_\bT\to\RR$, we define
    \begin{equation*}
        \delta^i_{\bu}\Xi_{\bs,\bt} := (\Id-\psi^i_\bu)\Xi_{\bs,\bt}
    \end{equation*}
    for any $\bs\le\bu\le\bt$. For an index-set $\theta\subset[d]$, we set
    \begin{equation*}
        \delta^\theta_\bu := \prod_{i\in\theta}\delta^i_\bu\,.
    \end{equation*}
\end{defn}

\noindent As this definition is rather technical, let us write out all the terms explicitly for $d=2$:
\begin{align*}
    \delta_\bu^{(1)}\Xi_{\bs,\bt} &= \Xi_{\bs,\bt}-(\Xi_{\bs,(u_1,t_2)}+\Xi_{(u_1,s_2),\bt})\\
    \delta_\bu^{(2)}\Xi_{\bs,\bt} &= \Xi_{\bs,\bt}-(\Xi_{\bs,(t_1,u_2)}+\Xi_{(s_1,u_2),\bt})\\
    \delta_\bu^{(1,2)}\Xi_{\bs,\bt} &= \Xi_{\bs,\bt}-(\Xi_{\bs,(u_1,t_2)}+\Xi_{(u_1,s_2),\bt})-(\Xi_{\bs,(t_1,u_2)}+\Xi_{(s_1,u_2),\bt})\\
    &\qquad+ (\Xi_{\bs,\bu}+\Xi_{(s_1,u_2),(u_1,t_2)}+\Xi_{(u_1,s_2),(t_1,u_2)}+\Xi_{\bu,\bt})
\end{align*}
We illustrated these terms for $\bu = \frac 12(\bs+\bt)$, where $\Xi_{\bs,\bt}$ is represented by an orange rectangle.

\begin{figure}
     \centering
     \begin{subfigure}[b]{0.3\textwidth}
        \centering
         \includegraphics[scale=1.5]{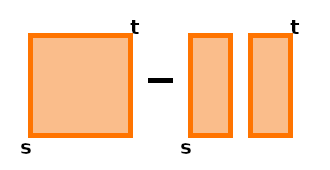}
         \caption{$\delta_{\bu}^{\{1\}}\Xi_{\bs,\bt}$}
     \end{subfigure}
     \begin{subfigure}[b]{0.3\textwidth}
         \centering
         \includegraphics[scale=1.5]{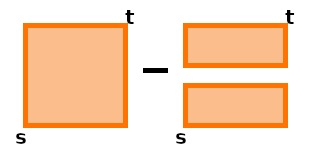}
         \caption{$\delta_{\bu}^{\{2\}}\Xi_{\bs,\bt}$}
     \end{subfigure}
     \\[\baselineskip]
     \begin{subfigure}[c]{0.3\textwidth}
         \centering
         \makebox[0pt]{\includegraphics[scale=1.5]{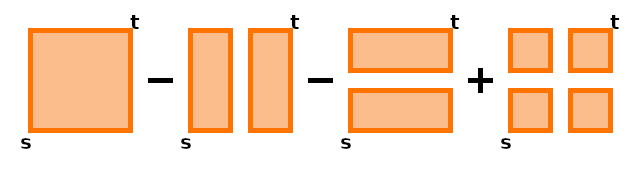}}
         \caption{$\delta_{\bu}^{\{1,2\}}\Xi_{\bs,\bt}$}
    \end{subfigure}
    \caption{$\delta$-notation}
\end{figure}

In dimension $d=1$, the $\delta$-operator has a particular interaction with so-called additive functions: A two-parameter function $f_{s,t}$ is of increment form, i.e. $f_{s,t} = f_t-f_s$ for some one-parameter function $f_t$, if and only if
\[
\delta_u f_{s,t} = 0\,.
\]
This observation extends to the case $d\ge 2$ in the sense that functions that satisfy the property 
\[
\delta_\bu^\theta F_{\bs,\bt}=0, 
\]
for all $\theta\subset[d]$ are precisely the rectangular increments of other $d$-parameter functions. That is, $\delta_\bu^\theta F_{\bs,\bt}=0$ for all $\theta\subset[d]$ if and only if there is some function $f:[\boo,\bT]\rightarrow \RR$ such that $F_{\bs,\bt}=\square^{[d]}_{\bs,\bt}f$. This gives rise to the following definition:
\begin{defn}
    We say that a multiparameter function $f:[\boo,\bT]\times[\boo,\bT]\rightarrow \RR$ is additive, if for all $\bs\leq \bu \leq \bt$ 
    \[
    \delta_\bu^{\theta} f_{\bs,\bt}=0\,
    \]
    holds for all $\theta\subset[d]$.
\end{defn}

\begin{rem}\label{rem:additive integral}
    In light of equation \eqref{eq:motivatingIntegrals}, a function given on (multiparameter) integral form, i.e.  $F_{\bs,\bt}= \int_{\bs}^\bt f_{\br} \dd \br $ is additive. 
\end{rem}

\begin{rem}
    Since $\delta^\theta_\bu = \prod_{i\in\theta}\delta^i_\bu$, it follows that $f$ is additive, if and only if for all $\bs\le\bu\le\bt$ in $[\boo,\bT]$ and $i\in [d]$, $\delta^i_\bu f_{\bs,\bt} = 0$.
\end{rem}
\medskip

\subsection{Grid-like partitions}

A partition of an intervall $[s,t]\subset\RR$ is a finite collection of points
\begin{equation*}
    \cP = \{s=p_1<p_2<\dots<p_n=t\}\,.
\end{equation*}
We identify each partition with the set of intervals
\begin{equation*}
    \cP = \{[p_1,p_2],\dots,[p_{n-1},p_n]\}\,.
\end{equation*}
Let $\cP^1,\dots, \cP^d$ be partitions of $[s_i,t_i]$, respectively. Given two points $\bs,\bt\in \RR$ and partitions $\cP^i$ of $[s_i,t_i]$ for $i\in[d]$, we define the grid-like partitions to be given by
\begin{equation*}
    \cP^\theta = \cQ^1\times\dots\times\cQ^d
\end{equation*}
for each $\theta\subset[d]$, where $\cQ^i = \cP^i$ for $i\in\theta$ and $\cQ^i = \{s_i,t_i\}$ else. With an abuse of notation, we can write this as
\begin{equation*}
    \cP^\theta = \prod_{i\in\theta}\cP^i\times[\bs_{\theta^c},\bt_{\theta^c}]\,,
\end{equation*}
where we use the notation $[\bs_{\theta^c},\bt_{\theta^c}] = \prod_{i\in\theta^c} [s_i,t_i]$. This is, of course, informal, as the product between sets is not commutative, so $\prod_{i\in\theta} \cP^i$ is not well defined. Even worse, multiplying $[\bs_{\theta^c},\bt_{\theta^c}]$ should not come at the right-hand-side, but needs to be interwoven in the product over $\theta$. Nevertheless, the above expression gives the right picture of what $\cP^\theta$ is supposed to be. We will therefore use the above notation to describe grid-like partitions.

As before, we identify grid-like partitions with sets of multiparameter intervals. Let $\cP^i = \{s_i = p^i_1\le\dots\le p^i_{n_i} = t_i\}$ and set $\bn = (n_1,\dots, n_d)$. Then we can identify $\cP^{[d]}$ with the set
\begin{equation*}
    \cP^{[d]} = \{[\bp_\bk,\bp_{\bk+[d]}]~\vert~ [d]\le \bk\le\bn-[d]\}
\end{equation*}
where $\bp_\bk = (p^1_{k_1},\dots,p^d_{k_d})$. For $\cP^\theta$, we simply replace the partitions $\cP^i$ with $\{s_i,t_i\}$ for $i\notin\theta$. Thus, $\cP^\theta$ is a partition of $[\bs,\bt]$ for all $\theta\subset[d]$. We define the mesh sizes of these partitions as
\begin{align*}
    \abs{\cP^\theta} := \sup_{[u,v]\in\bigcup_{i\in[d]}\cQ^i}\abs{u-v}\,,
\end{align*}
where $\cQ^i = \cP^i$ for $i\in\theta$ and $\cQ^i= \{[s_i,t_i]\}$ else, as before. Given an indexed set of partitions $\cP^i$ of some intervals $[s_i,t_i]$, we will always denote by $\cP^\theta$ the corresponding grid-like partition of $[\bs,\bt]$ and vice versa.

Given a two-parameter function $\Xi:[\bs,\bt]^2\to \RR$ and a grid-like partition $\cP$ of $[\bs,\bt]$, $\cP$ acts on $\Xi$ as follows:
\begin{equation*}
    \cP\Xi_{\bs,\bt} := \sum_{[\bu,\bv]\in\cP}\Xi_{\bu,\bv}\,.
\end{equation*}
If $\Xi$ is defined on one interval $[\bs,\bt]$ and $\cP$ is a grid-like partition of some other interval $[\bS,\bT]$, we can construct a grid-like partition of $[\bs,\bt]$ as follows: Given any $u_i\in[S_i,T_i]$, we set
\begin{equation}
\begin{split}
    \tilde u_i = \begin{cases}
        s_i &, u_i\le s_i \\
        u_i &, s_i<u_i\le t_i \\
        t_i &, t_i\le u_i\,.
    \end{cases}    
\end{split}
\label{tilde notation}
\end{equation}
and $\cP^i(s_i,t_i) := \{\tilde u_i ~\vert~ u_i\in\cP^i\}\cup\{s_i,t_i\}$. For a multiparameter interval $[\bs,\bt]$ and $\mathbf u\in\RR^d$, we accordingly set $\tilde {\mathbf u} := (\tilde u_1,\dots,\tilde u_d)$. We then construct the corresponding grid-like partition $\cP^\theta(\bs,\bt) := \prod_{i\in\theta} \cP^i(s_i,t_i)\times[\bs_{\theta^c},\bt_{\theta^c}]$, which is a partition of $[\bs,\bt]$.

The last bit of notation we introduce is the notation of \emph{neighbors} of a point $\bu$. Let $\bs\le\bu\le\bt$ and $\cP$ a grid-like partition of $[\bs,\bt]$. We call $u_i^-, u_i^+\in\cP^i$ the largest (respective smallest) point in $\cP$ such that $u_i^-\le u_i \le u_i^+$. We call $\bu^- := (u_1^-,\dots, u_d^-)$ and $\bu^+ := (u_1^+,\dots,u_d^+)$ the neighbors of $\bu$ in $\cP$.

\subsection{Multiparameter filtration and BDG inequality}

Throughout this paper, we assume $(\Omega,\cF,\PP)$ to be a complete probability space. In the multiparameter setting, there is no clear "past" of any point $\bx\in[\boo,\bT]$, so one must treat filtrations more carefully. Since we have a partial ordering $\le$ on $[\boo,\bT]$, we can use the theory of filtrations with partially ordered parameter sets \cite{Walsh86}:
\begin{defn}\label{def:multiparameter_filtration}
    We say that $(\mathcal F_\bt)_{\bt\in[\boo,\bT]}$ is a multiparameter filtration, if for every $\bs\le\bt$, it holds that $\cF_\bs\subset\cF_\bt$.
\end{defn}
\noindent We will assume that every multiparameter filtration $(\cF_\bt)_{\bt\in[\boo,\bT]}$ appearing throughout this paper is complete. The one-dimensional sewing lemma uses a BDG-type inequality (\cite{Le20}, equation (2.4)) as its main working horse: Given a discreet one-parameter filtration $(\cF_n)_{n\in\NN}$ and a sequence of random variables $(Z_k)_{k\in\NN}$ such that $Z_k$ is $\cF_{k+1}$-measurable, we have that
\begin{equation*}
    \norm{\sum_{k=1}^N Z_k}_m\lesssim \sum_{k=1}^N \norm{\EE[Z_k\vert\cF_k]}_m + \left(\sum_{k=1}^N \norm{Z_k-\EE[Z_k\vert\cF_k]}_m^2\right)^{\frac 12}\,.
\end{equation*}
Since we do not care for the constant in $\lesssim$, we use Jensen's inequality to bound $\norm{Z_k-\EE[Z_k\vert\cF_k]}_m\le 2\norm{Z_k}_m$. The above inequality thus simplifies to
\begin{equation}\label{ineq:BDG1d}
    \norm{\sum_{k=1}^N Z_k}_m\lesssim \sum_{k=1}^N \norm{\EE[Z_k\vert\cF_k]}_m + \left(\sum_{k=1}^N \norm{Z_k}_m^2\right)^{\frac 12}\,.
\end{equation}
The generalization of this equation into the multiparameter setting needs additional assumptions on the filtration $(\cF_\bs)_{\bs\in[\boo,\bT]}$, which we want to motivate by considering the two-dimensional case: Assume we are given a finite set of random variables $Z_{\bk}\in L_m$ for some $m\ge 2$ and $\bk\in\NN^2\cup[\boo,\bT]$, such that $Z_{\bk}$ is $\cF_{\bk+(1,1)}$-measurable. We want to bound $\norm{\sum_{\bk} Z_\bk}_m$. The simple idea is to decompose the sum
\begin{equation}\label{eq:Z_K_decomposition}
    \sum_{\bk} Z_\bk = \sum_{k_1=1}^{N_1}\sum_{k_2=1}^{N_2} Z_{(k_1,k_2)} = \sum_{k_1 = 1}^{N_1} \tilde Z_{k_1}
\end{equation}
and using \eqref{ineq:BDG1d} twice. To do so, we need \emph{marginal} filtrations $\cF_{k_1}^1,\cF_{k_2}^2$, such that
\begin{itemize}
    \item $\tilde Z_{k_1}$ is $\cF_{k_1+1}^1$-measurable.
    \item $\EE(Z_{(k_1,k_2)}\vert\cF^{(1)}_{k_1+1})$ is $\cF_{k_2+1}^2$-measurable.
\end{itemize}
Since $\tilde Z_{k_1} = \sum_{k_2} Z_{(k_1,k_2)}$, the first property implies that we should use
\begin{equation*}
    \cF^1_{k_1} = \sigma\left(\bigcup_{k_2} \cF_{(k_1,k_2)}\right)\,,
\end{equation*}
and for consistency (so that we can switch the sums and use the same construction), $\cF^2_{k_2} = \sigma\left(\bigcup_{k_1} \cF_{(k_1,k_2)}\right)$. For the second condition, we note that $Z_{(k_1,k_2)}$ is already $\cF_{k_2+1}^2$-measurable, so the condition should read that \emph{conditioning on $\cF^1_{k_1}$ does not interfere with conditioning on $\cF^2_{k_2}$}. Filtrations with this property are well-known in the literature as \emph{commuting filtrations}, for reference, see\cite{Imkeller85}, \cite{Walsh86}, \cite{Khoshnevisan02}.

\begin{defn}
    We say that a filtration $\cF_\bt$ is {\em commuting} if for all $\bs,\bt\in [\boo,\bT]$ and all bounded $\cF_\bt$ measurable random variables $Y$ we have 
    \begin{equation}\label{eq:commuting}
        \EE[Y|\cF_\bs]=\EE[Y|\cF_{\bs\wedge \bt}]
    \end{equation}
    \label{commuting def}
\end{defn}
\noindent As in the motivating example, we want to define marginal filtrations:
\begin{defn}
    For $\theta\subset[d], \bt\in[0,\bT]$, we set
    \begin{equation*}
        \cF^\theta_\bt := \sigma\left(\bigcup_{\bs\in[0,\bT]} \cF_{\pi^\theta_\bt \bs}\right)\,.
    \end{equation*}
    For $i\in[d]$, we call $\cF^i_\bt := \cF^{\{i\}}_\bt$ the marginal filtrations of $\cF_\bt$.
\end{defn}
\noindent Note that $\cF_\bt^\theta$ only depends on $(t_i)_{i\in\theta}$. $\cF_\bt^\theta$ can thus be seen as a multiparameter filtration over $\prod_{i\in\theta} [0,T_i]$. It especially follows that the marginal filtrations are one-parameter filtrations with $\cF^i_\bt = \cF^i_{t_i}$ only depending on $t_i$. Further note that $\cF^{[d]}_\bt = \cF_\bt$ and $\cF^\emptyset_\bt = \cF_\bT$. It easily follows that if $\cF_\bt$ is complete, so is $\cF_\bt^\eta$ for all $\eta\subset[d]$.

We use the notation $\EE^\eta_\bt X := \EE[X\vert\cF^\eta_\bt]$. The advantage of commuting filtrations is that they are fully characterized by their marginal filtrations, as the next result shows. It can be found for example in \cite{Walsh86}, page 349 or Theorem 3.4.1 in \cite{Khoshnevisan02}
\begin{thm}
    $\cF$ is a commuting multiparameter filtration, if and only if conditioning on the marginal filtrations commutes, i.e. for all $i,j=1,\dots,d$ and bounded random variables $X\in L_1(\Omega)$, we have
    \begin{equation*}
        \EE^i_\bt \EE^j_\bs X = \EE^j_\bs \EE^i_\bt X\,.
    \end{equation*}
    Furthermore, we have
    \begin{equation*}
        \EE^1_\bt\dots\EE^d_\bt X = \EE^{[d]}_\bt X\,.
    \end{equation*}
\end{thm}
\begin{cor}
    For all $\emptyset\neq\eta\subset[d]$, $\cF_\bt^\eta$ is a commuting filtration as a multiparameter filtration over $\prod_{i\in\eta} [0,T_i]$. It especially follows that
    \begin{equation}\label{eq:comFiltrations}
        \prod_{i\in\eta}\EE^i_\bt X = \EE^\eta_\bt X.
    \end{equation}
\end{cor}
\begin{proof}
    One easily sees that the marginal filtrations of $\cF_\bt^\eta$ are given by $\cF^i_\bt$ for $i\in\eta$. Since conditioning on these commutes, $\cF^\eta_\bt$ commutes.
\end{proof}
\noindent Note that per definition, $\cF^\emptyset_\bt = \cF_\bT$, so we have that $\EE^\emptyset_\bt X = \EE[X\vert\cF_\bT] = X = \prod_{i\in\emptyset} \EE^i_\bt X$, as long as $X$ is $\cF_\bT$ measurable. Thus, \eqref{eq:comFiltrations} holds for all $\eta\subset[d]$.
\begin{example}
    Given a Brownian sheet $(B_\bt)_{\bt\in[\boo,\bT]}$, we have that the filtration
    \begin{equation*}
        \cF_\bt := \overline{\sigma\left(B_\bs\vert\bs\le\bt\right)}
    \end{equation*}
    is commuting and its marginal filtrations are given by
    \begin{equation*}
        \cF^i_\bt = \overline{\sigma\left(B_\bs\vert s_i\le t_i\right)}\,.
    \end{equation*}
    See \cite{Walsh86} for reference.
\end{example}
\noindent In general, one should think of $\cF_\bt$ as capturing the behavior of our process over the section $\{\bs\le\bt\}$, i.e. everything to the bottom-left of $\bt$ in two dimensions. In this case, $\cF^1$ captures all the information to the left of $\bt$ and $\cF^2$ captures all the information to the bottom of $\bt$. From this perspective, it is natural to assume that first projecting onto $L_2(\cF^1_\bt)$ and then onto $L_2(\cF^2_\bt)$ should give the same result as projecting onto $L_2(\cF_\bt)$. 

For completeness' sake, it should be noted that Walsh defines and heavily uses another multiparameter filtration, which we call the weak past of a process. In the above example, it would be given by
\begin{equation}\label{strongPast}
    \cF^*_\bt = \overline{\sigma(B_\bs\vert s_i\le t_i \text{ for at least one }i\in[d])}
\end{equation}
while in general, it is the filtration generated by the marginal filtrations $\cF^*_\bt = \sigma\left(\bigcup_{i\in[d]} \cF^i_\bt\right)$. It will however only play a minor role in our analysis. The two-dimensional case is illustrated.
\begin{figure}
     \centering
     \begin{subfigure}[b]{0.2\textwidth}
        \centering
         \includegraphics[scale=2]{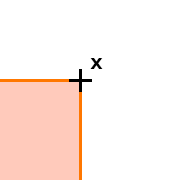}
         \caption{$\cF_\bx$}
     \end{subfigure}
     \begin{subfigure}[b]{0.3\textwidth}
        \centering
         \includegraphics[scale=2]{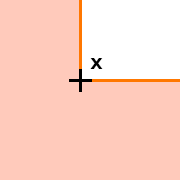}
         \caption{$\cF_\bx^*$}
     \end{subfigure}
     \begin{subfigure}[b]{0.2\textwidth}
         \centering
         \includegraphics[scale=2]{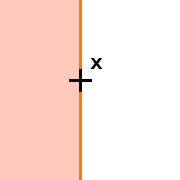}
         \caption{$\cF^1_\bx$}
     \end{subfigure}
     \begin{subfigure}[b]{0.2\textwidth}
         \centering
         \includegraphics[scale=2]{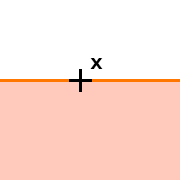}
         \caption{$\cF^2_\bx$}
     \end{subfigure}
        \caption{Strong past, weak past and marginal filtrations}
         \label{filtrations}
\end{figure}

To return to the construction \eqref{eq:Z_K_decomposition} for general dimension $d$, we note that the full result is rather technical and can be found in Lemma \ref{multiparameter BDG}. In this section, we only show the special case of vanishing conditional expectation, i.e. we assume that $\EE^i_\bk Z_\bk = 0$ for all $i\in [d]$, which is sufficient for our main results in Section \ref{sec_Local_Time}. In this case, \eqref{ineq:BDG1d} simplifies to \eqref{ineq:BDG1d} to $\norm{\sum_{i=1}^N Z_k}_m\lesssim \left(\sum_{k=1}^N\norm{Z_k}^2\right)^{\frac 12}$ and in $d\ge2$, we get:
\begin{lem}\label{lem:muti-parameterBDG-simplified}
    Let $I_i\subset\NN$ be finite sets wit minimal element $y_i\in I_i$. Set $\by = (y_1,\dots,y_d)$ and let $(\cF_\bn)_{\bn\in\NN^d}$ be a $d$-parameter filtration. Assume that $Z_\bk\in L_m(\Omega)$ for $m\ge 2, \bk\in I_{[d]}$ is $\cF_{\bk+[d]}$-measurable with
    \begin{equation*}
        \EE^i_\bk Z_\bk = 0
    \end{equation*}
    for all $i\in [d]$. Then
    \begin{equation}\label{eq:BDGsimplified1}
        \norm{\sum_{\bk\in I_\theta} Z_\bk}_m\lesssim\left(\sum_{\bk\in I_\theta}\norm{Z_\bk}_m^2\right)^{\frac12}
    \end{equation}
    and
    \begin{equation}\label{eq:BDGsimplified2}
        \EE_\by^i \sum_{\bk\in I_\theta} Z_\bk = 0
    \end{equation}
    holds for all $\theta\subset[d]$, and $i\in[d]$.
\end{lem}
\begin{proof}
    It is immediate to see \eqref{eq:BDGsimplified2}. To see \eqref{eq:BDGsimplified1}, we simply apply \eqref{ineq:BDG1d} $d$ times over the sum $\sum_{\bk\in I_\theta} Z_\bk$.
\end{proof}
\begin{rem}\label{rem:BDG_for_strong_past}
    It should be noted that there already exists a (different) generalization of the BDG-inequality \eqref{ineq:BDG1d} in \cite{Imkeller85}, Proposition 1: By using the weak past $\cF_\bt^* := \sigma\left(\bigcup_{i\in[d]}\cF^i_\bt\right)$, one can show that
    \begin{equation}\label{BDGImkeller}
    \norm{\sum_{\bk\in \{1,\dots,N\}^d} Z_{\bk}}_m\lesssim \sum_{\bk\in \{1,\dots,N\}^d} \norm{E[Z_\bk\vert\cF_\bk^*]}_m + \left(\sum_{\bk\in \{1,\dots,N\}^d} \norm{Z_\bk}_m^2\right)^{\frac 12}
    \end{equation}
    for bounded random variables $Z_\bk$ which are $\cF_{\bk+[d]}$-measurable. 

    This observation does however not fit our purpose, as it does not lead to the correct bounds in Corollary \ref{cor:sectorial_non_determinism}: For a fractional Brownian motion, conditional variance with respect to the strong past fulfills $\mbox{Var}(W_\bt\vert\cF_\bs^{\bepsilon})\gtrsim\sum_{i=1}^d \abs{t_i-s_i}^{2\zeta_i}$, whereas using the weak past $\cF_\bs^{\bepsilon,*}$ leads to the weaker estimate $\mbox{Var}(W_\bt\vert\cF_\bs^{\bepsilon,*})\gtrsim\min_{i=1,\dots, d} \abs{t_i-s_i}^{2\zeta_i}$. This is the main reason why we need a BDG inequality with respect to the strong past, given in Lemma \ref{lem:muti-parameterBDG-simplified}.
\end{rem}

\subsection{$C_2$ spaces}

A $C_2$ space is given by a stochastic 2d-parameter process $\Xi:\Delta_\bT\times\Omega\to \RR$. We say that $\Xi$ is adapted to $\cF_\bt$, if for all $(\bs,\bt)\in\Delta_\bT$ it holds that $\Xi_{\bs,\bt}$ is $\cF_\bt$-measurable. 

We need to generalize Hölder-norms to these processes. To do so, we use the following notation for $\boo\le\bx$ and $\bs,\bt\in[\boo,\bT]$:
\begin{align}
    \abs{\bt-\bs} &:= \sum_{i\in[d]}\abs{t_i-s_i}\\
    m_{\bx}(\bs,\bt) &:= \prod_{i\in[d]}\abs{t_i-s_I}^{x_i}\,.
\end{align}
Note that this implies that $m_{\theta}(\boo,\bt) = \prod_{i\in\theta}t_i$. We will also set $m_\bx(0,\alpha) := \prod_{i\in[d]}\alpha^{x_i}$ for $\alpha\in\RR$, and, since we identify $\theta\in\RR^d$ as a vector, $m_{\alpha\theta}(\bs,\bt) := \prod_{i\in\theta}\abs{t_i-s_i}^\alpha$. 

Let us take a short look at the deterministic case: In \cite{Harang2021}, the multiparameter sewing lemma requires that $\Xi_{\bs,\bt}$ has two regularities $0<\alpha\le\beta$, such that for any $\bs\le\br\le\bt$ and indexset $\emptyset\neq\theta\subset[d]$,
\begin{align*}
    \abs{\Xi_{\bs,\bt}}&\lesssim m_{\alpha}(\bs,\bt)\\
    \abs{\delta^\theta_\br\Xi_{\bs,\bt}} &\lesssim m_{\beta\theta+\alpha\theta^c}(\bs,\bt)\,.
\end{align*}
That is, $\Xi$ has Hölder-regularity $\alpha$ in all directions and if we apply the $\delta$-operator, the regularity increases to some $\beta\ge\alpha$. Sewing than requires $\beta>1$.

In the stochastic case, we want to replace $\abs\cdot$ with the $L_m$-norm $\norm{\cdot}_m$ for some $m\ge 2$ (so that we can apply the BDG-type inequality form lemma \ref{lem:muti-parameterBDG-simplified}) and assume that adding conditional expectation $\EE^i_\bs$ increases the regularity. In practice, applying $\EE^i_\bs$ should add another factor $\abs{t_i-s_i}^\gamma$ for some $\gamma>0$. It should then suffice that $\beta+\gamma>1$ to do stochastic sewing.

These ideas lead us to the following definition:

\begin{defn}
Assume $(\cF_\bt)_{\bt\in[\boo,\bT]}$ is a commuting filtration. Let $\Xi:\Delta_\bT\times\Omega\to\RR$, $\alpha,\in(0,1),\beta\in(\frac 12,\infty),\gamma\in(0,\frac 12)$ with $\beta>\alpha$. Let $m\ge 2$ and $\norm{\cdot}_m$ be the $L_m(\Omega)$- norm. We set
\begin{equation*}
    \norm{\Xi}_{\alpha,m} := \sup_{(\bs,\bt)\in\Delta_\bT} \frac{\norm{\Xi_{\bs,\bt}}_m}{m_{\alpha[d]}(\bs,\bt)}
\end{equation*}
and for two indexsets $\eta\subset\theta\subset[d]$, $\theta\neq\emptyset$ we set
\begin{equation*}
    \norm{\delta\Xi}_{\alpha,\beta,\gamma,m}^{\eta,\theta} := \sup_{\bs<\bu<\bt\in[0,\bT]} \frac{\norm{\EE^\eta_\bs\delta^\theta_\bu\Xi_{\bs,\bt}}_m}{m_{\beta\theta+\alpha\theta^c+\gamma\eta}(\bs,\bt)}\,.
\end{equation*}
We then set
\begin{equation*}
    \norm{\delta\Xi}_{\alpha,\beta,\gamma,m} := \sum_{\eta\subset\theta\subset[d],\theta\neq\emptyset} \norm{\delta\Xi}_{\alpha,\beta,\gamma,m}^{\eta,\theta}
\end{equation*}
and define
\begin{equation*}
    \norm{\Xi}_{\alpha,\beta,\gamma,m} := \norm{\Xi}_{\alpha,m} +\norm{\delta\Xi}_{\alpha,\beta,\gamma,m}\,.
\end{equation*}
$C^{\alpha,\beta,\gamma}_2 L_m((\cF_\bt)_{\bt\in[\boo,\bT]})$ is then given as the space of all $\cF_\bt$-adapted $\Xi_{\bs,\bt}$ in $L_m$, such that $\norm{\Xi}_{\alpha,\beta,\gamma,m} <\infty$.
\end{defn}

\begin{rem}
    In this paper, we will always assume that there is some commuting filtration $\cF_\bt$ in the background, and not explicitly refer to it unless necessary. In this sense, we abbreviate $C_2^{\alpha,\beta,\gamma} L_m = C_2^{\alpha,\beta,\gamma} L_m((\cF_\bt)_{\bt\in[\boo,\bT]})$.
\end{rem}

\noindent
Note that if we set $\gamma = 0$, we get the space of $\Xi$, such that the conditional expectation does not improve regularity. As this space will play an important role later, it gets its own definition:

\begin{defn}
    We denote by $C_2^{\alpha,\beta} L_m := C_2^{\alpha,\beta,0} L_m$ the space of adapted (with respect to some commuting filtration) $2d$-parameter processes $\Xi:\Delta_\bT\times\Omega\to\RR$, such that for all $\emptyset\neq\theta\subset[d]$
    \begin{align*}
        \norm{\Xi}_{\alpha,m} &<\infty\\
        \norm{\delta\Xi}_{\alpha,\beta,m}^\theta &:= \sup_{\bs<\bu<\bt\in[\boo,\bT]}\frac{\norm{\delta^\theta_\bu\Xi_{\bs,\bt}}_m}{m_{\beta\theta+\alpha\theta^c}(\bs,\bt)}<\infty\,.
    \end{align*}
    We call the norm of this space
    \begin{equation*}
        \norm{\Xi}_{\alpha,\beta,m} = \norm{\Xi}_{\alpha,m} + \sum_{\emptyset\neq\theta\subset[d]}\norm{\delta\Xi}_{\alpha,\beta,m}^\theta\,.
    \end{equation*}
\end{defn}

\noindent There is one further simplification one can make: If additionally to $\gamma=0$, it holds that $\alpha = \beta$, the restriction on $\norm{\delta\Xi}_{\alpha,\beta,m}^\theta$ for any $\emptyset\neq \theta\subset[d]$ becomes redundant. Indeed, since $\delta_\bu^\theta \Xi_{\bs,\bt}$ is a finite linear combination of $\Xi_{\bx,\by}$ where $\bx,\by$ are certain combinations of $\bs,\bu$ and $\bt$, we see that $\norm{\Xi}_{\alpha,m}<\infty$ immediately implies that $\norm{\delta\Xi}^\theta_{\alpha,\alpha,m}<\infty$. This observation leads to the last definition of this subsection:

\begin{defn}
    We denote by $C_2^{\alpha} L_m := C_2^{\alpha,\alpha,0} L_m$ the space of adapted (with respect to some commuting filtration) $2d$-parameter processes $\Xi:\Delta_\bT\times\Omega\to\RR$, such that
    \begin{equation*}
        \norm{\Xi}_{\alpha,m}<\infty\,.
    \end{equation*}
\end{defn}

\noindent In order to formulate our regularization by noise result, we will use the framework of Bessel potential spaces. For $\eta\in \RR$, we denote by $H^\eta$ the inhomogeneous Bessel potential space of order $\eta$, i.e.
\begin{equation}\label{bessel_potential}
H^\eta:=\left\{f\in \mathcal{S}'|\ \norm{f}_{H^\eta}=\norm{(1+|(\cdot)|)^\eta \hat{f}}_{L^2}<\infty \right\}.
\end{equation}

\section{Multiparameter stochastic sewing lemma}
\label{sewing section}

\noindent Let us briefly discuss multiparameter sewing before we introduce the stochastic multiparameter sewing lemma: As discussed in the introduction, the classical sewing lemma shows the convergence of Riemann-type sums
\begin{equation*}
    I^n_t = \sum_{[u,v]\in\cP^n([0,t])}A_{u,v}
\end{equation*}
as long as $\abs{\delta_{s,u,t} A}\lesssim \abs{t-s}^{1+\epsilon}$ for some $\epsilon > 0$. Furthermore, we can decompose any increment of the limit $I_t$ in to $A_{s,t}$ and a remainder $R_{s,t}$
\begin{equation*}
    I_t-I_s = A_{s,t} +R_{s,t}\,,
\end{equation*}
where $\abs{R_{s,t}}\lesssim \abs{t-s}^{1+\epsilon}$. This property uniquely characterizes $I_t$.

In higher dimension, $\delta_{s,u,t} A$ can be generalized to the $\delta$-operator $\delta^{[d]}_\bu\Xi_{\bs,\bt}$ for a $2d$-parameter process $\Xi$. But somewhat surprisingly, it does no longer suffice to ask 
\begin{equation*}
    \abs{\delta^{[d]}_\bu \Xi_{\bs,\bt}}\lesssim \prod_{i=1}^d \abs{t_i-s_i}^{1+\epsilon}
\end{equation*}
to show that $\sum_{[\bu,\bv]\in\cP_\bn}\Xi_{\bu,\bv}$ converges for any sequence of grid-like partitions $\cP_\bn$ with vanishing mesh size. As it turns out, one needs an appropriate bound for each $\delta^\theta_\bu\Xi_{\bs,\bt}$, $\theta\subset[d]$ to get convergence.

Let us at this point fix dimension $d=2$ for convenience's sake. In that case, one needs the bounds $\abs{\delta^1_\bu\Xi_{\bs,\bt}} \lesssim \abs{t_1-s_1}^{1+\epsilon}$, $\abs{\delta^2_\bu\Xi_{\bs,\bt}} \lesssim \abs{t_2-s_2}^{1+\epsilon}$ and $\abs{\delta^{(1,2)}_\bu\Xi_{\bs,\bt}} \lesssim \prod_{i=1}^2\abs{t_i-s_i}^{1+\epsilon}$ to get convergence of the Riemann-type sum, which we denote by $\cI^{(1,2)}\Xi_{\bs,\bt}$. This is, in fact, the square increment of the $d$-parameter process $\bt\mapsto \cI^{(1,2)}\Xi_{\boo,\bt}$, but for our purpose it is easier to think of the sewing of $\Xi$ as a $2d$-parameter process.

We want to uniquely characterize $\cI^{(1,2)}\Xi$. If we compare it to $\Xi$, we will find that the remainder fulfills
\begin{equation*}
    \abs{\cI^{(1,2)}\Xi_{\bs,\bt}-\Xi_{\bs,\bt}}\lesssim\sum_{i=1}^2\abs{t_i-s_i}^{1+\epsilon}\,,
\end{equation*}
which does in general not uniquely characterize $\cI^{(1,2)}\Xi$. The missing ingredient here are the \emph{partial sewings} of $\Xi$: Since we asked $\delta^1_\bu\Xi_{\bs,\bt}\lesssim\abs{t_i-s_i}^{1+\epsilon}$, we can fix the parameters $(s_2,t_2)$ and apply the one-dimensional sewing lemma to $(s_1,t_1)\mapsto\Xi_{(s_1,s_2),(t_1,t_2)}$ to construct the $(1)$-sewing
\begin{equation*}
    \cI^{1}\Xi_{\bs,\bt} = \lim_{n\to\infty}\sum_{[u_1,v_1]\in\cP_n} \Xi_{(u_1,s_2),(v_1,t_2)}\,.
\end{equation*}
This corresponds to fixing $(s_2,t_2)$ and constructing an integral of the first variable over the interval $[s_1,t_1]$. Analogously, we can construct the $(2)$-sewing $\cI^{2}\Xi_{\bs,\bt}$ by ``integrating'' over the interval $[s_2,t_2]$ while fixing the parameters $(s_1,t_1)$. Note that these terms are uniquely characterized by requiring that their remainders $R^i_{\bs,\bt} = \cI^i\Xi_{\bs,\bt}-\Xi_{\bs,\bt}$ fulfill
\begin{equation*}
    \abs{R^i_{\bs,\bt}}\lesssim \abs{t_i-s_i}^{1+\epsilon}
\end{equation*}
for $i=1,2$. As it turns out, $\Xi, \cI^1\Xi$ and $\cI^2\Xi$ together characterize the sewing $\cI^{(1,2)}$ uniquely: If we set the $(1,2)$-remainder to be
\begin{equation*}
    R^{(1,2)}_{\bs,\bt} = \cI^{(1,2)}\Xi_{\bs,\bt}-\cI^{1}\Xi_{\bs,\bt}-\cI^{2}\Xi_{\bs,\bt}+\Xi_{\bs,\bt}\,,
\end{equation*}
which mimics the construction of the $\delta$-operator, we find that $\abs{R^{(1,2)}_{\bs,\bt}}\lesssim\prod_{i=1}^2 \abs{t_i-s_i}^{1+\epsilon}$ holds and characterizes $\cI^{(1,2)}\Xi$ uniquely. So in the multiparameter case, we always get a family of sewings $\cI^\theta\Xi$ for all $\theta \subset[d]$, where $\theta$ denotes the parameters, over which we integrate, while the parameters in $\theta^c$ get fixed to $\bs_{\theta^c},\bt_{\theta^c}$ beforehand. 

More specifically, the multiparameter sewing lemma gives us a unique $2d$-parameter process $\cI^\theta\Xi_{\bs,\bt}$ for each $\theta\subset[d]$, such that

\begin{itemize}
    \item If we do not integrate over any index, we get $\Xi$, i.e. $\cI^\emptyset\Xi_{\bs,\bt} = \Xi_{\bs,\bt}$.
    
    \item All $\cI^\theta\Xi_{\bs,\bt}$ are additive in the indexes $i\in\theta$, i.e. $\delta^i_\bu\cI^\theta\Xi_{\bs,\bt} = 0$ for all $i\in\theta$ and $\bs\le\bu\le\bt$. Another way to think about this is that for any fixed $\bs_{\theta^c},\bt_{\theta^c}$, $(\bs_\theta,\bt_\theta)\mapsto \cI^\theta\Xi_{\bs,\bt}$ is a square increment $\square^\theta_{\bs_\theta,\bt_\theta} f(\cdot,\bs_{\theta^c},\bt_{\theta^c})$ for some $(2d-\abs{\theta})$-parameter process $f$. In particular, $\cI^{[d]}\Xi_{\bs,\bt}$ is a square increment of a $d$-parameter process. 

    \item We define the $\theta$-remainder to be
    \begin{equation*}
        R^\theta_{\bs,\bt} := \sum_{\eta\subset\theta} (-1)^{\eta}\cI^\eta\Xi_{\bs,\bt} = \Xi_{\bs,\bt} + (-1)^{\abs{\theta}}\cI^{\theta}\Xi_{\bs,\bt}+\sum_{\emptyset\neq\eta\subsetneq\theta}(-1)^{\eta}\cI^\eta\Xi_{\bs,\bt}\,.
    \end{equation*}
    It then holds that all $\cI^\theta\Xi_{\bs,\bt}$ are uniquely characterized by
    \begin{equation*}
        \abs{R^\theta_{\bs,\bt}}\lesssim\prod_{i\in\theta}\abs{t_i-s_i}^{1+\epsilon}\,,
    \end{equation*}
    for $\emptyset\neq\theta\subset[d]$.
\end{itemize}

\noindent How can one combine this with stochastic sewing? Recall that the main technique of stochastic sewing in dimension $d=1$ is to replace the assumption that $\abs{\delta_{s,u,t} A} \le \abs{t-s}^{1+\epsilon}$ with the the two assumptions
\begin{align*}
    \norm{\delta_{s,u,t} A}_m &\lesssim \abs{t-s}^{\frac 12+\epsilon_1}\\
    \norm{\EE_s \delta_{s,u,t} A}_m &\lesssim \abs{t-s}^{1+\epsilon_2}\,,
\end{align*}
for some $\epsilon_1,\epsilon_2>0$ and $m\ge 2$, where $\norm{\cdot}_m$ denotes the $L_m(\Omega)$-norm and $\EE_s = \EE[~\cdot~\vert\cF_s]$. That is, as long the conditional expectation $\EE_s\delta_{s,u,t} A$ has as much regularity as we needed in the deterministic sewing lemma, it suffices for $\norm{\delta_{s,u,t} A}_m$ to have half as much regularity. The price one pays for this is twofold: First, the Riemann-type sum $\sum_{[u,v]\in\cP^n([0,t])} A_{u,v}$ no longer converges almost surely but in the $L_m$-norm. And secondly, if we denote $R_{s,t} = I_t-I_s-A_{s,t}$ the remainder as before, we need two regularity conditions on $R_{s,t}$ to uniquely characterize $I_t$:
\begin{align*}
    \norm{R_{s,t}}&\lesssim\abs{t-s}^{\frac 12+\epsilon_1}\\
    \norm{\EE_s R_{s,t}}&\lesssim\abs{t-s}^{1+\epsilon_2}\,.
\end{align*}
As discussed in previous sections, in the multiparameter case, there is no unique concept of a past, which means that the filtration one can use is somewhat ambiguous. As it turns out, one does indeed have a choice here: Whether one wants to use the weak or strong past given by $\cF^*_\bs$ (see \eqref{strongPast}) and $\cF_\bs$ respectively. Let us first discuss the weak past $\cF^*_\bs$: As discussed in Remark \ref{rem:BDG_for_strong_past}, there is a BDG inequality that needs adaptedness with respect to the filtration $\cF_\bs$ of some $(Z_\bk)_{\bk\in\NN^d}$ to show
\begin{equation*}
    \norm{\sum_{\bk}Z_k}_m\lesssim \left(\sum_{\bk}\norm{Z_\bk-\EE^*_\bk Z_\bk}_m^2\right)^{\frac 12} + \sum_{\bk} \norm{\EE^*_\bk Z_\bk}_m\,.
\end{equation*}
If one uses this in the sewing lemma, one gets the three conditions
\begin{itemize}
    \item $\Xi_{\bs,\bt}$ is $\cF_\bt$ measurable for each $\bs\le\bt$.
    \item For all $\emptyset\neq\theta\subset[d]$, $\bs,\le\bu\le\bt$, We have the two regularities
    \begin{align}
        \norm{\delta_{\bu}^\theta\Xi_{\bs,\bt}}_m &\lesssim \prod_{i\in\theta}\abs{t_i-s_i}^{\frac 12+\epsilon_1}\label{ineq:coherence_condition1}\\
        \norm{\EE^{\theta,*}_\bs\delta_{\bu}^\theta\Xi_{\bs,\bt}}_m &\lesssim \prod_{i\in\theta}\abs{t_i-s_i}^{\frac 12+\epsilon_1}\lesssim \prod_{i\in\theta}\abs{t_i-s_i}^{1+\epsilon_2}\,,\label{ineq:coherence_*_condition}
    \end{align}
    where $\EE^{\theta,*}_\bs = \EE[\cdot~\vert\cF^{\theta,*}_\bs]$ denotes the conditional expectation on the sigma algebra $\cF^{\theta,*}_\bs = \sigma\left(\bigcup_{i\in\theta} \cF_\bs^i\right)$. Note that one has in particular $\cF^{[d],*}_\bs = \cF^*_\bs$.
\end{itemize}
It is actually possible to show that the Riemann-type sum $\sum_{[\bu,\bv]\in\cP_\bn}\Xi_{\bu,\bv}$ converges. While this setting is arguably more elegant than the sewing lemma we end up proving, $\norm{\EE^{\theta,*}_\bs\delta_{\bu}^\theta\Xi_{\bs,\bt}}_m$ is simply not the object we want to analyze: As mentioned in Remark \ref{rem:BDG_for_strong_past}, using $\cF^*_\bs$ to show local non-determinism does in general lead to weaker bounds of the conditional covariance $\mbox{Var}(X_\bt\vert\cF_\bs^*)$.

To achieve the bounds we hope for, we need to replace $\cF^*_\bs$ with the strong past $\cF_\bs$. This naturally leads to the question, of whether one can simply replace the $\cF^{\theta,*}_\bs$ with $\cF_\bs^\theta$ in \eqref{ineq:coherence_*_condition} and still get convergence of the Riemann sums. Unfortunately, the answer to this is no, since the BDG-inequality \eqref{BDGImkeller} only holds for the weak past $\cF^*_\bs$.
However, if we enrich \eqref{ineq:coherence_*_condition} with additional assumptions on $\EE^\eta_\bs \delta^\theta_\bu\Xi_{\bs,\bt}$ for all $\emptyset\neq\eta\subseteq\theta$, we can use the 1d-BDG type inequality inductively to show a new BDG-type inequality, which only uses the strong past $\cF_\bs$, as showed in Lemma \ref{multiparameter BDG}. This leads to the desired sewing lemma, which we state as Lemma \ref{lem:StochSewing}. In fact, if we replace $\cF_\bs^{\eta*}$ with $\cF_\bs^\eta$, one can combine \eqref{ineq:coherence_condition1} and \eqref{ineq:coherence_*_condition} into a single condition by interpreting \eqref{ineq:coherence_condition1} as the case $\eta = \emptyset$. If one does this, we get that the Riemann-sums converge, as long as
\begin{itemize}
    \item $\Xi_{\bs,\bt}$ is $\cF_\bt$-measurable and
    \item For all $\eta\subset\theta\subset[d]$, $\theta\neq\emptyset$ and $\bs\le\bu\le\bt$, we have that
    \begin{equation*}
        \norm{\EE^\eta_\bs\delta^\theta_\bu\Xi_{\bs,\bt}}_m\lesssim \prod_{i\in\theta\setminus\eta}\abs{t_i-s_i}^{\frac12+\epsilon_1}\prod_{i\in\eta}\abs{t_i-s_i}^{1+\epsilon_2}\,.
    \end{equation*}
\end{itemize}
These conditions are what one would intuitively expect: $\delta^\theta_{\bs,\bt}$ has regularity $\beta>\frac 12$, but whenever we condition on the marginal $\sigma$-field $\cF^i_\bs$, the regularity in the direction $i$ increases to a $\gamma>1$. Under these conditions, one gets convergence of the Riemann sums to objects $\cI^\theta\Xi_{\bs,\bt}$ in the $L_m(\Omega)$-norm for all $\theta\subset[d]$, as in the deterministic case. Furthermore, the $\cI^\theta\Xi_{\bs,\bt}$ are uniquely characterized under the same conditions as before if one replaces the bound on $R^\theta_{\bs,\bt}$ by
\begin{equation*}
    \norm{\EE^\eta_\bs R^\theta_{\bs,\bt}}_m \lesssim \prod_{i\in\theta\setminus\eta}\abs{t_i-s_i}^{\frac 12+\epsilon_1}\prod_{i\in\eta}\abs{t_i-s_i}^{1+\epsilon_2}\,.
\end{equation*}
Since the proof of the sewing lemma is highly technical, we will present it in the simplified case in which we assume that all conditional expectations $\EE^i_\bs\delta^\theta_\bu\Xi_{\bs,\bt}$ vanish for all $i\in\theta$, $\emptyset\neq\theta\subset[d]$ and $\bs\le\bu\le\bt$. The statement of the general stochastic multiparameter sewing lemma, as well as the respective BDG-type inequality necessary for the general proof, can be found in Subsection \ref{sec:general_sewing_lemma}.

\subsection{A technical result on multiparameter sequences}

Before we show the proof of the stochastic multiparameter sewing lemma, let us discuss one of the key techniques of (deterministic) multiparameter sewing: Let us fix dimension $d=2$, and consider a sequence of dyadic grid-like partitions $\cP_\bn$ of the interval $[\bs,\bt]$. We want to show that
\begin{equation}\label{eq:convergentSequnece}
    \cP_\bn\Xi_{\bs,\bt} = \sum_{[\bu,\bv]\in\cP_\bn}\Xi_{\bu,\bv} 
\end{equation}
converges. In the one-dimensional case, one does that by showing that $\norm{(\cP_n-\cP_{n+1})\Xi_{\bs,\bt}}\lesssim 2^{-n\epsilon}$ for some $\epsilon>0$ holds. It then follows, that $\cP_\bn\Xi_{\bs,\bt}$ is a Cauchy sequence in $L_m$, which implies convergence.
In the $2$-dimensional case, it is slightly more tricky: While $\norm{(\cP_{\bn}-\cP_{(n_1+1,n_2+1)})\Xi_{\bs,\bt}}\lesssim 2^{-n_1\epsilon}+2^{-n_2\epsilon}$ would imply convergence along the sequence $(\bn+k\boldsymbol 1)_{k\in\NN}$, it is in general hard to show this kind of bound.

To get a better overview, let us simplify the problem even further and try to ``cut apart'' a single interval $[\bs,\bt]$ at a point $\bs\le\bu\le\bt$. Recall that the ``cut apart''-operators are defined by $\psi^1_\bu \Xi_{\bs,\bt} = \Xi_{\bs, (u_1,t_2)} + \Xi_{(u_1,s_2),\bt}$ and analogously, $\psi^2_\bu\Xi_{\bs,\bt} = \Xi_{\bs,(t_1,u_2)}+\Xi_{(s_1,u_2),\bt}$. If we want to pass from $\cP_\bn$ to $\cP_{(n_1+1,n_2+1)}$, we have to cut apart intervals of $\cP_\bn$ along both directions, so we will end up with objects of the form $\psi^1_\bu\psi_\bu^2\Xi_{\bs,\bt} = \psi^{(1,2)}_\bu\Xi_{\bs,\bt}$. So what we are actually trying to bind from above are terms of the form
\begin{equation}\label{term:Psi}
    \norm{\Xi_{\bs,\bt}-\psi^{(1,2)}_\bu \Xi_{\bs,\bt}}_m\,.
\end{equation}
However, in most situations, we only control the delta-operator $\norm{\delta_\bu^{(1,2)}\Xi_{\bs,\bt}}_m$. So our goal is to show that \eqref{eq:convergentSequnece} is Cauchy, as long as $\norm{\delta^\eta_\bu\Xi_{\bs,\bt}}_m$ is small enough for all $\emptyset\neq\eta\subset[d]$. Observe that we can decompose the term \eqref{term:Psi} into a sum over delta operators via 
\begin{equation*}
    \Id-\psi^{(1,2)}_\bu = (\Id- \psi^1_\bu) + (\Id - \psi_\bu^2) - (\Id-\psi_\bu^1)(\Id-\psi_\bu^2) = \delta^1_\bu + \delta^2_\bu - \delta^{\{1,2\}}_\bu
\end{equation*}
\begin{figure}
    \centering 
    \includegraphics[scale=1.5]{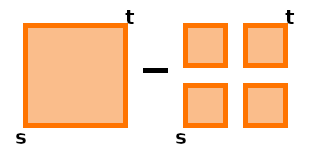}
    \caption{$\Xi_{\bs,\bt}-\psi^{(1,2)}_\bu\Xi_{\bs,\bt}$}
\end{figure}
Thus we can bound $\Id-\psi^{(1,2)}_\bu$ as long as the delta operators are bounded in a sufficient way. Analogously, we can show that the sequence $(\cP_\bn\Xi_{\bs,\bt})_{\bn\in\NN^d}$ is Cauchy, as long as we have sufficient bounds on $\norm{(\Id-I^1)\cP_\bn\Xi_{\bs,\bt}}_m$, $\norm{(\Id-I^2)\cP_\bn\Xi_{\bs,\bt}}_m$ and $\norm{(\Id-I^1)(\Id-I^2)\cP_\bn\Xi_{\bs,\bt}}_m$, where the operators $I^1\cP_\bn = \cP_{(n_1+1,n_2)}$ and $I^2\cP_\bn = \cP_{(n_1,n_2+1)}$ increase the respective index by one. We formalize this in Lemma \ref{lem:technical2}, which uses some new notation:

\begin{itemize}
    \item For an indexset $\theta\subset[d]$ and a multiindex $\bn\in\NN^d$, $\bn_\theta = \pi_\bn^\theta \boo$ is the multiindex given by $n_i$ for $i\in\theta$ and $0$, else.
    \item For a multiindex $\bn\in\NN^d$ and a set of rates $\balpha>\boo$, we set $\bn\balpha := (n_1\alpha_1,\dots,n_d\alpha_d)$ to be the pointwise product.
\end{itemize}

\begin{lem}\label{lem:technical2}
Let $(h_\bn)_{\bn\in\NN^{d}}$ be a multiparameter sequence in a normed space $\cH$. We define the operators $I^ih_\bn = h_{\bn+\{i\}}$ Assume that there is a $C>0$ and rates $\balpha = (\alpha_1,\dots,\alpha_d)>\boo$ such that for all $\bn\in\NN^d$ and $\emptyset\neq\theta\subset[d]$
\begin{equation}\label{ineq:Assumptionh}
    \norm{\prod_{i\in\theta} (\Id- I^i)h_{\bn_\theta}} \le Cm_{\balpha\bn_\theta}\left(0,\frac 12\right)\,,
\end{equation}
Then it holds that for all $\bn\le\bk$ and $\theta\subset[d]$
\begin{equation*}
    \norm{h_{\bn_\theta}-h_{\bk_\theta}}\lesssim C\sum_{i\in\theta}2^{-n_i\alpha_i}\,,
\end{equation*}
where the constant in $\lesssim$ might depend on $\balpha$. It especially follows that $(h_{\bn})$ is a Cauchy sequence.
\end{lem}

\begin{proof}
    The proof consists of 3 steps: One first shows that our bound on $\norm{\prod_{i\in\theta}(\Id-I^i)h_{n_\theta}}$ implies a similar bound on $\norm{\prod_{i\in\eta}(\Id-I^i)h_{n_\theta}}$ for all $\emptyset\neq\eta\subset\theta$. In the second step, we introduce the operator $I^i_{k_i}h_\bn := (I^i)^{k_i}h_\bn = h_{\bn+k_i\{i\}}$ and bound the term $\norm{\prod_{i\in\eta}(\Id- I^i_{k_i})h_{\bn_\theta}}$ for all $k_i\ge 1$, $i=1,\dots, d$. In the third step, we finally conclude that this implies a bound on $\norm{(\Id-I^\eta_\bk)h_{\bn_\theta}}$ where $I^\eta_\bk h_\bn = \left(\prod_{i\in\eta}I^i_{k_i}\right)h_\bn = h_{\bn+\bk_\eta}$, which concludes the proof.

\bigskip

    \noindent \emph{Step 1:} We claim that for any index-sets $\emptyset\neq\eta\subset\theta\subset[d]$, we have
    \begin{equation*}
        \norm{\prod_{i\in\eta}(\Id-I^i)h_{\bn_\theta}}\lesssim Cm_{\balpha\bn_\eta}\left(0,\frac 12\right)\,.
    \end{equation*}
    This can be shown by induction over $\abs{\theta\setminus\eta}$. For $\theta = \eta$, this is just \eqref{ineq:Assumptionh}. For $\eta\subsetneq\theta$, let $j\in\theta\setminus\eta$ and use the decomposition
    \begin{equation*}
        \prod_{i\in\eta}(\Id-I^i)h_{\bn_\theta} = \prod_{i\in\eta}(\Id-I^i)h_{\bn_{\theta\setminus\{j\}}} - \sum_{k=0}^{n_j-1}\prod_{i\in\eta\cup\{j\}}(\Id-I^i)h_{{\pi^j_k\bn_\theta}}\,.
    \end{equation*}
    The induction hypothesis  then implies
    \begin{equation*}
        \norm{\prod_{i\in\eta}(\Id-I^i)h_{\bn_\theta}}\lesssim \left(1+\sum_{k=0}^{n_j-1}2^{-k\alpha_j}\right)Cm_{\balpha\bn_\eta}\left(0,\frac 12\right)\,,
    \end{equation*}
    which shows the claim (recall that the constant in $\lesssim$ is allowed to depend on $\balpha$). 

\bigskip

    \noindent\emph{Step 2:} Let $\eta,\theta$ as in step 1, and assume that $k_i\ge 1$ for all $i\in\eta$. We use the decomposition $(\Id-I^i_{k_i}) = \sum_{l_i = 0}^{k_i-1}(\Id-I^i)I^i_{l_i}$ and calculate directly
    \begin{align}
        \norm{\prod_{i\in\eta} (\Id-I^i_{k_i})h_{\bn_\theta}} &= \norm{\prod_{i\in\eta} \sum_{l_i = 0}^{k_i-1}(\Id-I^i)I^i_{l_i}h_{\bn_\theta}}\nonumber\\
        &\le \sum_{\bl_\eta=\boo}^{\bk_\eta-\eta}\norm{\prod_{i\in\eta}(\Id-I^i)h_{\bn_\theta+\bl_\eta}}\nonumber\\
        &\lesssim \left(\sum_{\bl_\eta = \boo}^{\bk_\eta-\eta}m_{\balpha\bl_\eta}\left(0,\frac 12\right)\right) Cm_{\balpha\bn_\eta}\left(0,\frac12\right)\\
        &\lesssim  Cm_{\balpha\bn_\eta}\left(0,\frac12\right)\,.\label{ineq:prodKI}
    \end{align}

\bigskip

    \noindent\emph{Step 3:} We claim that for all $\eta\subset\theta\subset[d]$ as well as multiindexes $\bn,\bk$ with $k_i\ge 1$ for all $i\in\eta$, it holds that
    \begin{equation*}
        \norm{h_{\bn_\theta}-h_{\bn_\theta+\bk_\eta}} = \norm{(\Id-I^\eta_\bk)h_{\bn_\theta}}\lesssim C\sum_{i\in\eta}2^{-n_i\alpha_i}\,,
    \end{equation*}
    which proofs the lemma by setting $\eta = \theta$. The claim will be shown via induction over $\abs{\eta}$. The case $\eta = \emptyset$ is clear. For non-empty $\eta$, observe that 
    \begin{equation*}
        \prod_{i\in\eta}(\Id-I^i_{k_i}) = \sum_{\tilde\eta\subset\eta}(-1)^{\abs{\tilde\eta}} I^{\tilde\eta}_{\bk} = \sum_{\tilde\eta\subset\eta}(-1)^{\abs{\tilde\eta}} (I^{\tilde\eta}_{\bk}-\Id)\,,
    \end{equation*}
    where we used that for all finite sets $\theta$, $\sum_{\eta\subset\theta}(-1)^{\abs\eta} =0$. Rearranging those terms allows us to calculate
    \begin{align}\begin{split}\label{calc3technicalLema}
        \norm{(\Id-I^\eta_\bk)h_{\bn_\theta}} &\le \norm{\prod_{i\in\eta}(\Id-I^i_{k_i})h_{\bn_\theta}} + \sum_{\tilde\eta\subsetneq \eta}\norm{(\Id-I^{\tilde\eta}_{\bk})h_{\bn_\theta}} \\
        &\lesssim C\sum_{i\in\eta}2^{-n_i\alpha}\,,
    \end{split}\end{align}
    by the result of step 2 and the induction hypothesis.
\end{proof}

\subsection{Stochastic multiparameter sewing in the simplified setting}

While we aim to prove the multiparameter stochastic sewing lemma in the most general setting possible, we also want to present a simplified, yet very practical version: In Section \ref{sec_Local_Time}, we will apply the sewing lemma to the two-parameter process $\Xi_{\bs,\bt} = \EE_\bs^{[d]}\int_{[\bs,\bt]} X(\br) d\br$ for a certain stochastic process $X:[0,\bT]\times\Omega\to\RR$. By the tower property and using that our filtration is commuting, it is easy to see that all expected values $\EE_\bs^\eta \delta^\theta_\br\Xi_{\bs,\bt} = 0$ vanish for all $\emptyset\neq\eta\subset\theta\subset[d]$. This drastically simplifies technical parts of the sewing lemma, while keeping all its ideas. Therefore, we present the proof in this setting.

\begin{lem}[Stochastic multiparameter sewing, simplified version]\label{lem:stochastic sewing simplified}
    Let $\Xi\in C_2^{\alpha,\beta}L_m$ for $\alpha\in(0,1)$, $\beta\in(\frac 12,\infty)$ with $\beta\geq\alpha$ and assume that for all $\emptyset\neq\eta\subset\theta\subset[d]$, it holds that
    \begin{equation}
        \EE^\eta_\bs\delta^\theta_\br\Xi_{\bs,\bt} = 0\label{eq:expDelta=0}
    \end{equation}
    for all $\bs\le\br\le\bt$. Further, let $\cP_\bn^{[d]} = \prod_{i\in[d]} \cP^i_{n_i}$ be the $\bn$-th dyadic partition of $[\bs,\bt]$. Then the sequence $(\cP^\theta_\bn\Xi_{\bs,\bt})_\bn$ given by 
    \begin{equation*}
     \cP^\theta_\bn\Xi_{\bs,\bt}:=\sum_{[\bu,\bv]\in\cP^\theta_\bn}\Xi_{\bu,\bv}
    \end{equation*}
    converges in $L_m(\Omega)$ as $\bn\to \infty$. We denote the corresponding limit by $ \cI^\theta\Xi_{\bs,\bt}$, which we also call $\theta$-sewing of the germ $\Xi$. In the case $\theta=[d]$, we call $\mathcal{I}^{[d]}\Xi_{\bs, \bt}$ the sewing of $\Xi$.  Furthermore, $(\cI^\theta\Xi)_{\theta\subset[d]}$ is the unique family of two-parameter processes, such that:
    \begin{enumerate}[label = (\roman*)]
        \item $\cI^\theta \Xi_{\bs,\bt}$ is $\cF_\bt$-measurable for all $\bs\le\bt$.
        \item $\cI^\emptyset\Xi = \Xi$
        \item $\cI^\theta\Xi$ is a additive in $[\boo,\bT_\theta]$, i.e. $\delta^\eta_\bu\cI^\theta\Xi_{\bs,\bt} = 0$ for all $\bs\le\bu\le\bt$, $\emptyset\neq\eta\subset\theta\subset[d]$.
        \item $\exists C>0$ such that the following holds for all $\theta\subset[d]$ and $\bs\le\bt$:
        \begin{equation}
            \norm{\sum_{\theta'\subset\theta} (-1)^{\abs {\theta'}}\cI^{\theta'}\Xi_{\bs,\bt}}_m \le C\norm{\Xi}_{\alpha,\beta,m}m_{\alpha\theta^c+\beta\theta}(\bs,\bt)\,,\label{eq:UniquePropertySimplified1}
        \end{equation}
        and for all $\emptyset\neq\eta\subset\theta$, we have
        \begin{equation}
            \EE^\eta_\bs\sum_{\theta'\subset\theta} (-1)^{\abs{\theta'}}\cI^{\theta'}\Xi_{\bs,\bt} = 0\,.\label{eq:UniquePropertySimplified2}
        \end{equation}
    \end{enumerate}
\end{lem}

\begin{rem}
    It should be noted that there already exists a stochastic reconstruction theorem \cite{kern2023stochastic}, which is able to deal with dimension $d>1$. But unlike the one-dimensional case, the stochastic multiparameter sewing lemma is not a special case of the stochastic reconstruction theorem. This is due to the fact, that most stochastic processes have better-behaved square increments than increments, and the multiparameter sewing lemma is better suited to deal with the square increments. Indeed, consider the construction of the integral $\int_{\bs}^\bt B_\br dX_\br$ for a deterministic process $X$ and a Brownian sheet $B$. Then the stochastic reconstruction theorem will require that the conditional expectation
    \begin{align*}
        \EE^1_\bx((B_\by - B_\bx)\square_{\by,\bz} X)\\
        \EE^2_\bx((B_\by - B_\bx)\square_{\by,\bz} X)
    \end{align*}
    vanishes for $\bx\le\by\le\bz$ (or at least has good enough $L_m$-norm), and will thus fail. The stochastic multiparameter sewing lemma on the other hand needs
    \begin{gather*}
        \EE^1_\bx(\delta^1_{\bx,\by} B\square_{(y_1,x_2),\bz} X)\\
        \EE^2_\bx(\delta^2_{\bx,\by} B\square_{(x_1,y_2),\bz} X)\\
        \EE^\eta_\bx(\square_{\bx,\by} B\square_{\by,\bz} X)
    \end{gather*}
    to vanish for all $\emptyset\neq\eta\subset\{1,2\}$, which it indeed does. In general, if one is interested in constructing stochastic integrals, our lemma seems to be better suited to deal with situations in which the stochastic properties come from the integrand rather than the integrator.
\end{rem}

\noindent
Just as in most sewing lemmas, it holds that $\cI^\theta\Xi_{\bs,\bt} = L_m-\lim_{\bn\to\infty}\sum_{[\bu,\bv]\in\cP^\theta_\bn}\Xi_{\bu,\bv}$ not just for the dyadic partitions $\cP_\bn$, but for all series of grid-like partitions, such that the mesh $\abs{\cP_\bn}$ vanishes as $\bn\to\infty$. The following lemma makes this rigorous:

\begin{lem}\label{lem:every partition converges}
    Let $\Xi$ be as in Lemma \ref{lem:stochastic sewing simplified}. Let $\cQ_\bn$ be any sequence of grid-like partitions of the interval $[\bs,\bt]$ with mesh $\abs{\cQ_\bn}\to 0$ as $\bn\to\infty$. Then it follows, that
    \begin{equation*}
        \cQ^\theta_\bn\Xi_{\bs,\bt} = \sum_{[\bu,\bv]\in\cQ_\bn^\theta}\Xi_{\bu,\bv} \xrightarrow{\bn\to\infty} \cI^\theta\Xi_{\bs,\bt},
    \end{equation*}
converges in $L_m$ for all $\theta\subset[d]$.
\end{lem}

\noindent In fact, we will use an even more restrictive case later on, in which $\alpha = \beta$ holds. In this case, it suffices to find a bound for $\norm{\Xi_{\bs,\bt}}_m$ instead of $\norm{\delta^\theta_\bu\Xi_{\bs,\bt}}_m$ for all $\emptyset\neq\theta\subset[d]$, simplifying the sewing lemma even further. Recall that $C^\alpha_2 L_m$ is the space of adapted (with respect to some commuting filtration) 2d-parameter processes $\Xi_{\bs,\bt}$, such that $\norm{\Xi_{\bs,\bt}}_m\lesssim m_{\alpha[d]}(\bs,\bt)$.

\begin{cor}\label{cor:sewing_alpha=beta}
    Let $\alpha>\frac 12$ and $\Xi_{\bs,\bt}\in C^\alpha_2 L_m$. Assume that
    \begin{align*}
        \EE^\eta_\bs\delta_\br^\theta\Xi_{\bs,\bt} &= 0
    \end{align*}
    holds for all $\bs\le\br\le\bt$ and $\emptyset\neq\eta\subset\theta\subset[d]$. Then for any $\theta\subset[d]$ the Riemann-series $(\cP^\theta_\bn\Xi_{\bs,\bt})_\bn$  given by
    \begin{equation*}
       \cP^\theta_\bn\Xi_{\bs,\bt}= \sum_{[\bu,\bv]\in\cP^\theta_\bn} \Xi_{\bu,\bv}
    \end{equation*}
    converges for all grid-like partitions $\cP_\bn$ with vanishing mesh size $\abs{\cP_\bn}$ to the unique limit  $\mathcal{I}^{\theta}\Xi_{\bs, \bt}$ (independent of $(\cP_\bn)_{\bn\in\NN^d}$), as $\bn$ goes to $\infty$. As before, the limits $(\cI^\theta\Xi_{\bs,\bt})_{\theta\subset[d]}$ are the unique family of multiparameter processes fulfilling properties $(i) - (iii)$ from Lemma \ref{lem:stochastic sewing simplified} as well as

    \begin{itemize}
        \item[(iv')] $\exists C>0$ such that for all $\theta\subset[d]$ and $\bs\le\bt$, we have
        \begin{equation}\label{ineq:uniqueness23}
            \norm{\cI^\theta \Xi_{\bs,\bt}}_m \le C\norm{\Xi}_{\alpha,m}m_{\alpha[d]}(\bs,\bt)\,,
        \end{equation}
        and for all $\emptyset\neq\eta\subset\theta$, we have
        \begin{equation}
            \EE^\eta_\bs\sum_{\theta'\subset\theta} (-1)^{\abs{\theta'}}\cI^{\theta'}\Xi_{\bs,\bt} = 0\,.\label{eq:uniqueness23}
        \end{equation}
    \end{itemize}
\end{cor}

\noindent Note that this is almost precisely Lemma \ref{lem:stochastic sewing simplified} with $\alpha = \beta$, except that \eqref{eq:UniquePropertySimplified1} got replaced by \eqref{ineq:uniqueness23}. So for the proof, it suffices to show that these two inequalities are equivalent, as long as $(i)-(iii)$ hold. 

\begin{proof}
    It is immediate to see that \eqref{ineq:uniqueness23} implies \eqref{eq:UniquePropertySimplified1} in the case $\alpha = \beta$. The converse follows by induction over $\abs\theta$: Assume that $(i)-(iii)$ as well as \eqref{eq:UniquePropertySimplified1} holds. Then $\norm{\cI^\emptyset\Xi_{\bs,\bt}}_m = \norm{\Xi_{\bs,\bt}}_m \le \norm{\Xi}_{\alpha,m} m_{\alpha[d]}(\bs,\bt)$. For non-trivial $\theta$, we decompose

    \begin{align*}
        \norm{\cI^\theta \Xi_{\bs,\bt}}_m &\le \norm{\sum_{\theta'\subset\theta}(-1)^{\abs{\theta'}} \cI^{\theta'}\Xi_{\bs,\bt}}_m + \sum_{\theta'\subsetneq\theta}\norm{(-1)^{\abs{\theta'}} \cI^{\theta'}\Xi_{\bs,\bt}}_m\\
        &\lesssim C\norm{\Xi}_{\alpha,m} m_{\alpha[d]}(\bs,\bt)\,,
    \end{align*}

    \noindent where we use \eqref{eq:UniquePropertySimplified1} and the induction hypothesis.
\end{proof}

\noindent We can now tackle the proof of Lemma \ref{lem:stochastic sewing simplified} and \ref{lem:every partition converges}:

\begin{proof}[Proof of lemma \ref{lem:stochastic sewing simplified}]

    \textit{Step 1: Existence}

    \noindent
    Let us start by showing the convergence of the sequence $\cP_\bn\Xi_{\bs,\bt}$ for dyadic grid-like partitions $\cP_\bn$ of $[\bs,\bt]$. We define the operator $I^i \cP^\theta_\bn = \cP^\theta_{\bn+\{i\}}$. The main argument is the observation, that
    \begin{equation}\label{eq:deltas}
        \prod_{i\in\theta}(Id-I^i)P^\theta_{\bn}\Xi = \sum_{[\bu,\bv]\in\cP^\theta_\bn} \left(\Id-\psi_{\frac{\bu+\bv}{2}}^\theta\right) \Xi_{\bu,\bv} = \sum_{[\bu,\bv]\in\cP^\theta_\bn} \delta^\theta_{\frac{\bu+\bv}2} \Xi_{\bu,\bv}
    \end{equation}
    holds. Note that $\delta^\theta_{\frac{\bu+\bv}2} \Xi_{\bu,\bv}$ is $\cF_\bv$-measurable and has conditional expectation $\EE^\eta_\bu \delta^\theta_{\frac{\bu+\bv}2} \Xi_{\bu,\bv} = 0$ for all $\emptyset\neq \eta\subset\theta$ by assumption, so we can use Lemma \ref{lem:muti-parameterBDG-simplified} to calculate
    \begin{align*}
        \norm{\prod_{i\in\theta}(Id-I^i)P^\theta_{\bn}\Xi}_m &= \norm{\sum_{[\bu,\bv]\in\cP^\theta_\bn} \delta^\theta_{\frac{\bu+\bv}2} \Xi_{\bu,\bv}}_m\\
        &\lesssim \left(\sum_{[\bu,\bv]\in\cP^\theta_\bn} \norm{\delta^\theta_{\frac{\bu+\bv}2} \Xi_{\bu,\bv}}_m^2\right)^{\frac12} \\
        &\lesssim \norm{\Xi}_{\alpha,\beta,m}\left(\sum_{[\bu,\bv]\in\cP^\theta_\bn} m_{\alpha\theta^c + \beta\theta}(\bu,\bv)^2\right)^{\frac 12}
    \end{align*}
    We use that $\abs{u_i-v_i} = 2^{-n_i}\abs{t_i-s_i}$ for $i\in\theta$ and $\abs{u_i-v_i} = \abs{t_i-s_i}$ for $i\notin\theta$, as well as the fact that the sum consist of $m_{\bn_\theta}(0,2)$ many terms, to conclude
    \begin{equation}\label{ineq:technicalCalc}
        \norm{\prod_{i\in\theta}(Id-I^i)P^\theta_{\bn}\Xi_{\bs,\bt}}_m\lesssim \norm{\Xi}_{\alpha,\beta,m}m_{-\bn_\theta(\beta-\frac 12)}(0,2) m_{\alpha\theta^c+\beta\theta}(\bs,\bt)\,.
    \end{equation}
    Lemma \ref{lem:technical2} now immediately implies that for $\bn\le\bk$
    \begin{equation*}
        \norm{(\cP^\theta_\bn-\cP^\theta_\bk)\Xi_{\bs,\bt}}\lesssim\norm{\Xi}_{\alpha,\beta,m}m_{\alpha[d]}(\bs,\bt)\sum_{i\in\eta}2^{-n_i(\beta-\frac 12)}
    \end{equation*}
    and that $\cP^\theta_\bn\Xi_{\bs,\bt}$ is a Cauchy sequence for all $\emptyset\neq\theta\subset[d]$. Since $L_m(\Omega)$ is a Banach space, it is thus convergent.

\bigskip
    
    \noindent
    \textit{Step 2: Properties (i)-(iv)}
    It is immediate that $\cI^\theta\Xi_{\bs,\bt}$ is $\cF_\bt$-measurable, and that $\cI^\emptyset\Xi =\Xi$. For the property (iv), we use that 
    \begin{equation*}
        \prod_{i\in\theta}(Id-I^i_{n_i}) P_{\boo}^\theta = \sum_{\theta'\subset\theta}(-1)^{\abs{\theta'}}P_\bn^{\theta'}\,.
    \end{equation*}
    If we look at the proof of Lemma \ref{lem:technical2} with $h_{\bn_\theta} = P^\theta_\bn\Xi_{\bs,\bt}$, we see that calculation \eqref{ineq:prodKI} together with \eqref{ineq:technicalCalc} implies
    \begin{align*}
        \norm{\sum_{\theta'\subset\theta}(-1)^{\abs{\theta'}}P_\bn^{\theta'}\Xi_{\bs,\bt}} &\lesssim \norm{\Xi}_{\alpha,\beta,m}\left(\sum_{\bk_\theta = \boo}^{\bn_\theta-\theta}m_{-\bk_\theta(\beta-\frac 12)}\left(0,2\right)\right) m_{\alpha\theta^c+\beta\theta}(\bs,\bt)\\
        &\lesssim \norm{\Xi}_{\alpha,\beta,m}m_{\alpha\theta^c+\beta\theta}(\bs,\bt)
    \end{align*}
    Letting $\bn$ go to $\infty$ gives us \eqref{eq:UniquePropertySimplified1}. \eqref{eq:expDelta=0} together with the identity \eqref{eq:deltas} implies that $\EE_\bs^\eta\prod_{i\in\theta}(\Id-I^i)\cP^\theta_{\bn}\Xi_{\bs,\bt} = 0$ for all $\emptyset\neq\eta\subset\theta\subset[d]$, which immediately gives \eqref{eq:UniquePropertySimplified2}.

    It remains to show that for all $\emptyset\neq\eta\subset\theta\subset[d]$, $\delta_\br^\eta\cI^\theta\Xi_{\bs,\bt} = 0$. Since $\delta^\eta_\br = \prod_{i\in\eta}\delta^i_{r_i}$, it suffice to show that $\delta_{r_i}^i\cI^\theta\Xi_{\bs,\bt} = 0$ for all $i\in\theta$. So let $r_i\in[s_i,t_i]$ and consider the grid-like partition
    \begin{equation*}
        \tilde \cP_\bn^\theta = \prod_{j\neq i}\cP^j_{n_j}\times (\cP^i_{n_i} \cup\{r_i\})\,.
    \end{equation*}
    We calculate that
    \begin{equation*}
        \norm{\left(P_{n_i}^i -\tilde P_{n_i}^i\right)\Xi_{\bs,\bt}}_m = \norm{\delta^i_{r_i}\Xi_{\pi^i_{r_i^-}\bs,\pi^i_{r_i^+}\bt}}_m\lesssim \norm{\Xi}_{\alpha,\beta,m}2^{-n_i\beta}\abs{t_i-s_i}^\beta m_{\alpha \{i\}^c}(\bs,\bt)\,,
    \end{equation*}
    where $r_i^-,r_i^+$ are the neighbors of $r_i$ in $\cP_{n_i}^i$. For all index-sets $\theta\subset[d]$ such that $i\in\theta$, this observation extends to
    \begin{equation*}
        \prod_{j\in\theta\setminus\{i\}} (\Id-I^j)(\cP^\theta_\bn-\tilde\cP^\theta_\bn)\Xi_{\bs,\bt} = \sum_{[\bu,\bv]\in\cP^{\theta\setminus\{i\}}_\bn}\delta^\theta_{\pi^i_{r_i}(\frac{\bu+\bv}2)} \Xi_{\pi^i_{r_i^-}\bu,\pi^i_{r_i^+}\bv}
    \end{equation*}
    which gives us with the use of Lemma \ref{lem:muti-parameterBDG-simplified}
    \begin{align*}
        \norm{\prod_{j\in\theta\setminus\{i\}} (\Id-I^j)(\cP^\theta_\bn-\tilde\cP^\theta_\bn)\Xi_{\bs,\bt}}_m &\lesssim \left(\sum_{[\bu,\bv]\in\cP^{\theta\setminus\{i\}}_\bn} \left(\norm{\Xi}_{\alpha,\beta,m}m_{-\beta\bn_\theta}(0,2)m_{\beta\theta+\alpha\theta^c}(\bs,\bt)\right)^2\right)^{\frac 12}\\
        &\le \norm{\Xi}_{\alpha,\beta,m}2^{-n_i\beta}m_{-(\beta-\frac 12)\bn_{\theta\setminus\{i\}}}(0,2)m_{\beta\theta+\alpha\theta^c}(\bs,\bt)\,.
    \end{align*}
    We apply Lemma \ref{lem:technical2} to the sequence $\bk\mapsto (\cP^\theta_{\tilde\bk}-\tilde\cP^\theta_{\tilde\bk})\Xi_{\bs,\bt})$, where $\bk\in\NN^{d-1}$ and $\tilde\bk = (k_1,\dots, k_{i-1}, n_i, k_i,\dots,k_{d-1})$. This gives us
    \begin{align*}
        \norm{(\cP^\theta_\bn-\tilde\cP^\theta_\bn) \Xi_{\bs,\bt}}_m &\le  \norm{(\cP^\theta_{\pi^i_\bn 0}-\tilde\cP^\theta_{\pi^i_\bn 0}) \Xi_{\bs,\bt}}_m + \norm{\left[(\cP^\theta_\bn-\tilde\cP^\theta_\bn) - (\cP^\theta_{\pi^i_\bn 0}-\tilde\cP^\theta_{\pi^i_\bn 0})\right] \Xi_{\bs,\bt}}_m\\
        &\lesssim \norm{\Xi}_{\alpha,\beta,m}2^{-n_i\beta}\abs{t_i-s_i}^\beta m_{\alpha \{i\}^c}(\bs,\bt)\xrightarrow{\bn\to\infty} 0\,.
    \end{align*}
    Thus, $\cI^\theta\Xi_{\bs,\bt} = \lim_{\bn\to\infty}\tilde\cP^\theta_\bn\Xi_{\bs,\bt}$. Using the notation $\delta_{x_i}^i\tilde\cP^\theta_\bn\Xi_{\bs,\bt} := \tilde\cP_\bn^\theta \delta_{x_i}^i\Xi_{\bs,\bt}$ for all $x_i\in[s_i,t_i]$, we see that $\delta^i_{x_i} \tilde\cP^\theta_\bn\Xi_{\bs,\bt} = 0$ for all $\bx\in\tilde\cP_\bn^\theta$. It especially follows that $\delta_{r_i}^i\tilde\cP^\theta_\bn\Xi_{\bs,\bt}=0$ and thus its limit fulfills $\delta^i_{r_i}\cI^\theta\Xi_{\bs,\bt} = 0$\,.

\bigskip

    \noindent
    \textit{Step 3: Uniqueness}
    Let $\theta\subset [d]$ and let $\cI^\theta\Xi,\tilde\cI^\theta\Xi$ be two $2d$-parameter processes fulfilling properties $(i)-(iv)$. Let $Z^\theta := \cI^\theta\Xi-\tilde\cI^\theta\Xi$. Then $Z^\theta$ fulfills: 
    \begin{itemize}
    \item $Z_{\bs,\bt}^\theta$ is $\cF_\bt$-measurable for all $\bs\le\bt$.
    \item $Z^\emptyset = 0$.
    \item $Z^\theta$ is a additive in $[\boo,\bT_\theta]$.
    \item $\exists C>0$ such that the following holds for all $\theta\subset[d]$ and $\bs\le\bt$:
    \begin{equation}
        \norm{\sum_{\theta'\subset\theta} (-1)^{\abs {\theta'}}Z^{\theta'}_{\bs,\bt}}_m \le C\norm{\Xi}_{\alpha,\beta,m}m_{\alpha\theta^c+\beta\theta}(\bs,\bt)\,,\label{eq:UniquePropertySimplifiedZ1}
    \end{equation}
    and for all $\emptyset\neq\eta\subset\theta$, we have
    \begin{equation}
        \EE^\eta_\bs\sum_{\theta'\subset\theta} (-1)^{\abs{\theta'}}Z^{\theta'}_{\bs,\bt} = 0\,.\label{eq:UniquePropertySimplifiedZ2}
    \end{equation}
    \end{itemize}
    We claim that these properties imply $Z^\theta = 0$, which we show by induction over $\abs{\theta}$. For $\theta = \emptyset$, this is true by the second property. For $\theta\neq\emptyset$, we use  the induction hypothesis to observe that for $\theta'\subsetneq \theta $

    \begin{equation}
         Z^{\theta'}_{\bs,\bt}=0
    \end{equation}
    and thus by the third property and Lemma \ref{lem:muti-parameterBDG-simplified} we may write
    \begin{align*}
        \norm{Z^\theta_{\bs,\bt}}_m &=\norm{\sum_{[\bu,\bv]\in\cP^\theta_\bn}Z^\theta_{\bu,\bv}}_m =  \norm{\sum_{[\bu,\bv]\in\cP^\theta_\bn}\sum_{\theta'\subset\theta}(-1)^{\abs{\theta}+\abs{\theta'}}Z^{\theta'}_{\bu,\bv}\,}_m\\
        &\lesssim \left(\sum_{[\bu,\bv]\in\cP^\theta_\bn} \norm{\sum_{\theta'\subset\theta} (-1)^{\abs{\theta'}} Z^{\theta'}_{\bu,\bv}}_m^2\right)^{\frac12}\\
        &\le C\norm{\Xi}_{\alpha,\beta,m}\left(\sum_{[\bu,\bv]\in\cP^\theta_\bn} m_{2\alpha\theta^c+2\beta\theta}(\bu,\bv)\right)^{\frac 12}\\
        & = C\norm{\Xi}_{\alpha,\beta,m}m_{\bn_\theta(\beta-\frac 12)}\left(0,\frac 12\right)m_{\alpha\theta^c+\beta\theta}(\bs,\bt)\xrightarrow{\bn\to\infty}0\,.
    \end{align*}
    Thus, $Z^\theta_\bn = 0$, which finishes the proof.
\end{proof}

\noindent
Let us now move to Lemma \ref{lem:every partition converges}, for which we, in particular, recall the notation in \eqref{tilde notation}.  The main idea of the proof is to show that for every partition $\cQ$ with small enough mesh size, there is a dyadic partition $\cP_\bn$ for some very high $\bn$, such that $\cQ\Xi_{\bs,\bt}$ and $\cP_\bn \Xi_{\bs,\bt}$ are very close in $L_m$-norm. The result then follows from the convergence of $\cP_\bn\Xi_{\bs,\bt}$.

During the proof, we will encounter terms of the form $\cP_\bn\Xi_{\bs,\bt}$, where $\cP_\bn$ is not a dyadic partition of $[\bs,\bt]$, but of some other interval $[\bS,\bT]$. The next lemma gives two useful properties of this expression:

\begin{lem}\label{lem:technical1}
    Let $ \Xi$ be as in Lemma \ref{lem:stochastic sewing simplified}. Let $\bS\le\bs\le\bt\le\bT$ and let $\cP_{\bn}$ the $\bn$-th dyadic partition of $[\bS,\bT]$. Let $\bn$ be such that $2^{-n_i}\abs{T_i-S_i}\ge\abs{t_i-s_i}\ge 2^{-n_i-1}\abs{T_i-S_i}$. We denote by
    \begin{equation*}
    \cP_{\bn}\Xi_{\bs, \bt} = \sum_{[\bu,\bv]\in \cP_n}\Xi_{\tilde \bu,\tilde \bv}\,.
    \end{equation*}
    It holds that for all $\theta\subset[d]$ and $\bn\le\bk$ and $j\in\theta$:
    \begin{align}
        \norm{\prod_{i\in\theta}(\Id-\cP^i_{k_i})\Xi_{\bs,\bt}}_m&\lesssim \norm{\Xi}_{\alpha,\beta,m}m_{\alpha\theta^c+\beta\theta}(\bs,\bt)\label{ineq:technical3}\\
        \EE^j_\bs \prod_{i\in\theta}(\Id-\cP^i_{k_i})\Xi_{\bs,\bt} &= 0\,.\label{eq:technical3}
    \end{align}
\end{lem}

\begin{proof}
    Observe that for any $\bk\ge\bn$, it holds that
    \begin{align*}
        \prod_{i\in\theta}(\Id-\cP^i_{k_i})\Xi_{\bs,\bt} &= \prod_{i\in\theta}\left[(\Id-\cP^i_{n_i})+\sum_{m= n_i}^{k_i-1}(\Id- I^i)\cP^i_m\right]\Xi_{\bs,\bt}\\
        &=\sum_{\eta\subset\theta}\sum_{\bm_\eta = \bn_\eta}^{\bk_\eta-\eta}\prod_{i\in\theta\setminus\eta}(\Id-\cP^i_{n_i})\prod_{i\in\eta}(\Id-I^i)\cP_{\bm_\eta}\Xi_{\bs,\bt}\,,
    \end{align*}
    where $I^iP_\bm^\theta = P_{\bm+\{i\}}^\theta$, as before. Analyzing the summands for each $\bm_\eta$ gives us
    \begin{equation*}
        \prod_{i\in\theta\setminus\eta}(\Id-P^i_{n_i})\prod_{i\in\eta}(\Id-I^i)\cP_{\bm_\eta}\Xi_{\bs,\bt} = \sum_{[\bu,\bv]\in\cP_{\bm_\eta}}\delta^\theta_{\tilde\br}\Xi_{\tilde\bu,\tilde\bv}\,,
    \end{equation*}
    where $r_i = \frac{u_i+v_i}2$ for $i\in\eta$ and for $i\in\theta\setminus\eta$, $r_i = t_i^-$ is the highest number $r_i<t_i$ in $\cP^i_{n_i}$ (note that by our choice of $n$, there is at most one point $r_i\in\cP^i_{n_i}$ such that $r_i\in [s_i,t_i]$). \eqref{eq:technical3} follows immediately. To see \eqref{ineq:technical3}, note that $\abs{\tilde u_i-\tilde v_i}\lesssim 2^{-(m_i-n_i)}\abs{t_i-s_i}$ holds for $i\in\eta$ and $\abs{\tilde u_i-\tilde v_i}\lesssim \abs{t_i-s_i}$ for $i\in\theta\setminus\eta$. It follows by Lemma \ref{lem:muti-parameterBDG-simplified}:
    \begin{align*}
        \norm{\prod_{i\in\theta\setminus\eta}(\Id-P^i_{n_i})\prod_{i\in\eta}(\Id-I^i)\cP_{\bm_\eta}\Xi_{\bs,\bt}}_m &\lesssim \left(\sum_{[\bu,\bv]\in\cP_{\bm_\eta}}\norm{\delta^\theta_{\tilde\br}\Xi_{\tilde\bu,\tilde\bv}}_m^2\right)^{\frac 12}\\
        &\lesssim \norm{\Xi}_{\alpha,\beta,m}m_{\alpha\theta^c+\beta\theta}(\bs,\bt)m_{-(\bm-\bn)_\eta(\beta-\frac 12)}(0,2)\,,
    \end{align*}
    where we used that the number of non-zero terms is of order $\prod_{i\in\eta}2^{m_i-n_i}$. Summing over $\bm_\eta\ge \bn_\eta$ gives \eqref{ineq:technical3}.
\end{proof}

\noindent We now have all ingredients for the proof of Lemma \ref{lem:every partition converges}: 

\begin{proof}[Proof of Lemma \ref{lem:every partition converges}]
    Let $\cP_\bk$ be the $\bk$-th dyadic partition of $[\bs,\bt]$ and assume that $\bk$ is large enough, such that at most one $\bu\in\cQ_\bn$ is in each $[\bu^-,\bu^+]\in\cP_\bk$. (Recall that $\bu^-,\bu^+$ are the neighbors of $\bu$ in $\cP_\bk$. This means that we allow $\bk$ to depend on $\bn$.) Let $\cA$ be the joint refinement of $\cP_\bk,\cQ_\bn$. By construction of $\cP_\bk$, we get that
    \begin{align*}
        \norm{\prod_{i\in\theta}(\cP^i_{k_i}-\cA^i)\Xi_{\bs,\bt}}_m &= \norm{\sum_{\bu\in\prod_{i\in\theta}\cQ^i_{n_i}\times\{\bs_\theta^c\}}\delta^\theta_\bu \Xi_{\bu_\theta^-,\bu_\theta^+}}_m\\
        &\lesssim \#\cQ_\bn m_{\alpha\theta^c+\beta\theta}(\bs,\bt)m_{-\bk_\theta\beta}(0,2)\,,
    \end{align*}
    where $\#\cQ_\bn$ is the number of points in $\cQ_n$, and $\bu^-_\theta$, $\bu^+_\theta$ denotes the neighbors of $\bu$ in $\cP^\theta_\bk$ and where we disregard $\norm{\Xi}_{\alpha,\beta,m}$.  It follows as before from calculation \eqref{calc3technicalLema} that
    \begin{equation*}
        \norm{(\cA^\theta-\cP^\theta_\bk)\Xi_{\bs,\bt}}_m \lesssim \sum_{i\in\theta}\#\cQ_\bn m_{\alpha[d]}(\bs,\bt) 2^{-k_i\beta}\,,
    \end{equation*}
    which shows that $\norm{(\cP^\theta_\bk-\cA^\theta)\Xi_{\bs,\bt}}<\epsilon$ for large enough $\bk$ for all $\theta\subset [d]$.
    It remains to show that $\norm{(\cA^\theta-\cQ_\bn^\theta)\Xi_{\bs,\bt}}_m$ is small. To do so, let $[\bu,\bv]\in \cQ^\theta_\bn$. Then by Lemma \ref{lem:technical1}, we get that
    \begin{align*}
        \norm{\prod_{i\in\theta}(\Id-\cA^i)\Xi_{\bu,\bv}}_m = \norm{\prod_{i\in\theta}(\Id-\cP_{k_i}^i)\Xi_{\bu,\bv}}_m &\lesssim m_{\alpha\theta^c+\beta\theta}    (\bu,\bv)\\
        \EE^i_\bu \prod_{i\in\theta}(\Id-\cA^i)\Xi_{\bu,\bv} = \EE^i_\bu \prod_{i\in\theta}(\Id-\cP_{k_i}^i)\Xi_{\bu,\bv} &= 0\,,
    \end{align*}
    for all $i\in\theta$. Using Lemma \ref{lem:muti-parameterBDG-simplified}, we now get that
    \begin{align*}
        \norm{\prod_{i\in\theta}(\cA^i-\cQ^i_{n_i})\Xi_{\bs,\bt}}_m &\lesssim \left(\sum_{[\bu,\bv]\in\cQ_\bn^\theta}\norm{\prod_{i\in\theta}(\Id-\cA^i)\Xi_{\bu,\bv}}_m^2\right)^{\frac 12}\\
        &\lesssim m_{\alpha\theta^c+\frac 12\theta}(\bs,\bt)\prod_{i\in\theta}\abs{\cQ^i_{n_i}}^{\beta-\frac 12}\,,
    \end{align*}
    where we used that $\left(\sum_{[u_i,v_i]\in\cQ_{n_i}^i} \abs{u_i-v_i}^{2\beta}\right)^{\frac 12} \lesssim \abs{Q_{n_i}^i}^{\beta-\frac 12}\abs{t_i-s_i}^{\frac 12}$. It follows that
    \begin{equation*}
        \norm{(\cA-\cQ_\bn)\Xi_{\bs,\bt}}\lesssim \abs{\cQ_\bn}^{\beta-\frac 12}m_{\min(\alpha,\frac 12)[d]}(\bs,\bt)\,,
    \end{equation*}
    which goes to zero for $\bn\to\infty$. Thus, we can choose an $\bk(\bn)$ which goes to $\infty$ as $\bn\to\infty$, such that $\norm{(\cP_\bk-\cQ_\bn)\Xi_{\bs,\bt}}$ vanishes. This shows the claim.
\end{proof}

\subsection{General setting}\label{sec:general_sewing_lemma}

As stated before, we also aim to show a general stochastic multiparameter sewing lemma. That is, we will not assume in this section, that the conditional expectations $\EE^i_\bs \delta^\theta_\bu\Xi_{\bs,\bt}$ vanish, but that they have an appropriate bound
\begin{equation*}
    \norm{\EE^\eta_\bs \delta^\theta_\bu\Xi_{\bs,\bt}}_m\lesssim\prod_{i\in\theta\setminus\eta}\abs{t_i-s_i}^{\frac 12+\epsilon_1}\prod_{i\in\eta}\abs{t_i-s_i}^{1+\epsilon_2}\,,
\end{equation*}
for all $\eta\subset\theta\subset[d]$, $\theta\neq\emptyset$ as well as $\bs\le\bu\le\bt$. The main difference in the proof of the sewing lemmas will be that we can no longer use the BDG-type inequality given in Lemma \ref{lem:muti-parameterBDG-simplified} and need to prove a new one. Recall that in the case of dimension $1$, we look at sequences $(Z_k)_{k\in\NN}$ such that $Z_k$ is $\cF_{k+1}$-measurable and the $L_m$-norm of its conditional expectation is better behaved than its $L_m$-norm, i.e. $\norm{\EE_kZ_k}_m\lesssim \norm{Z_k}_m$, the one-dimensional BDG-inequality gives us
\begin{equation*}
    \norm{\sum_{k=1}^N Z_k}_m\lesssim \sum_{k=1}^N\norm{\EE_kZ_k}_m + \left(\sum_{k=1}^N\norm{Z_k}_m^2\right)^{\frac 12}\,.
\end{equation*}
If we are in dimension $d=2$, that is we have a 2-parameter sequence $(Z_\bk)_{\bk\in\NN^2}$ of random variables in $L_m(\Omega)$, such that each $Z_\bk$ is $\cF_{\bk+\boldsymbol 1}$-measurable for a commuting filtration $(\cF_\bk)_{\bk\in\NN^2}$, the commuting property allows us to apply the one-dimensional BDG-inequality to $\tilde Z_{k_1} = \sum_{k_2=1}^N Z_{(k_1,k_2)}$. Afterward, one can apply the BDG-inequality a second time to get
\begin{align}
\begin{split}\label{term:sum_Z_k_not_possible}
    \norm{\sum_{\bk \in[N]^2} Z_\bk}_m &\lesssim \sum_{k_1=1}^N\norm{\EE^1_{k_1} \tilde Z_{k_1}}_m + \left(\sum_{k_1=1}^N \norm{\tilde Z_{k_1}}_m^2\right)^{\frac 12}\\
    &\lesssim \sum_{\bk\in[N]^2}\norm{\EE^{(1,2)}_\bk Z_\bk}_m + \sum_{k_1 = 1}^N \left(\sum_{k_2=1}^N \norm{\EE^1_{k_1}Z_\bk}_m^2\right)^{\frac 12} \\
    &\qquad +\left(\sum_{k_1=1}^N \left(\sum_{k_2=1}^N \norm{\EE_{k_2}^2 Z_\bk}_m + \left(\sum_{k_2=1}^N\norm{Z_\bk}_m^2\right)^{\frac 12}\right)^2\right)^{\frac 12}\\
    &\le \sum_{\bk\in[N]^2}\norm{\EE^{(1,2)}_\bk Z_\bk}_m + \sum_{k_1 = 1}^N \left(\sum_{k_2=1}^N \norm{\EE^1_{k_1}Z_\bk}_m^2\right)^{\frac 12} \\
    &\qquad +\left(\sum_{k_1=1}^N \left(\sum_{k_2=1}^N \norm{\EE_{k_2}^2 Z_\bk}_m \right)^2\right)^{\frac 12} + \left(\sum_{\bk\in[N]^2}\norm{Z_\bk}_m^2\right)^{\frac 12}
\end{split}
\end{align}
This forms a rather complicated algebraic structure since we need to keep track of where exactly the powers $2$ and $\frac 12$ appear in the nested sums. It still holds that for each index $i\in[d]$ in each term, we either sum over squares or have a conditional expectation $\EE^i_{k_i}$, so we still expect to see the improvement of regularity which makes stochastic sewing possible. However, the above formula is already somewhat impractical for dimension $d=2$ and will only get more complicated for higher dimensions $d$.

We can sight-step these algebraic difficulties thanks to the following observation: Eventually, we want to apply the $BDG$ inequality to sums over $\delta_\bu^\theta \Xi_{\bs,\bt}$ for a $2d$-parameter process $\Xi$, $\bs\le\bu\le\bt$ and an index-set $\emptyset\neq\theta\subset[d]$. The $L_m$ norm of the conditional expectation of these terms factorizes in the sense, that for each $\eta\subset\theta$ we have:
\begin{equation*}
    \norm{\EE^\eta_\bs\delta_\bu^\theta\Xi_{\bs,\bt}}_m \lesssim \prod_{i\in\theta\setminus\eta}\abs{t_i-s_i}^{\frac 12+\epsilon_1} \prod_{i\in\eta}\abs{t_i-s_i}^{1+\epsilon_2}\,.
\end{equation*}
 If we replace $\delta^\theta_\bu\Xi_{\bs,\bt}$ with $Z_\bk$ and call  $a_{i,k_i} := \abs{t_i-s_i}^{\frac 12+\epsilon_1}$ as well as $b_{i,k_i} := \abs{t_i-s_i}^{1+\epsilon_2}$, we get the following property of $Z_\bk$: There are constants $0< b_{i,k_i}, a_{i,k_i},c$ such that
\begin{equation*}
    \norm{\EE^\eta_\bk Z_\bk}_m\lesssim c\cdot b_{\eta,\bk}\cdot a_{\eta^c,\bk}\,,
\end{equation*}
where $b_{\eta_\bk} := \prod_{i\in\eta} b_{i,k_i}$ and $a_{\eta^c,\bk} := \prod_{i\in\eta^c} a_{i,k_i}$. Under this additional assumption, \eqref{term:sum_Z_k_not_possible} can be neatly written as a sum over all subsets $\theta\subset[d]$
\begin{equation*}
    \norm{\sum_{\bk\in[N]^2} Z_\bk}_m \lesssim \sum_{\theta\subset[2]}\left(\sum_{\bk_{\theta^c}\in[N]^{\theta^c}} b_{\theta^c,\bk}\right)\left(\sum_{k_\theta\in[N]^{\abs\theta}} a_{\theta,\bk}^2\right)^{\frac 12}\,,
\end{equation*}
where one can think of $\theta$ as the set of indices, for which the $k_i$ get summed over squares in \eqref{term:sum_Z_k_not_possible}, while for the indices $i\notin\theta$, the $k_i$ are summed over conditional expectations $\EE^i_{k_i}$. This leads us to the following BDG-type inequality, where we use the notation that $\dot\cup$ denotes disjoint unions:
\begin{lem}\label{multiparameter BDG}
    For each $i\in\{1,\dots,d\}$, let $I_i\subset\NN$ be a finite set with minimal element $y_i$, and set $I_\theta := \prod_{i\in\theta} I_i\times\prod_{j\notin\theta}\{y_j\}$ for each index set $\theta$, such that all $I_\theta\subset\NN^d$. Further, let $(\mathcal F_\bn)_{\bn\in\NN^d}$ be a commuting $d$-parameter filtration and let $Z_\bk\in L_m(\Omega)$ for each $\bk\in I_{[d]}$. We assume that $Z_\bk$ is $\cF_{\bk+ \boldsymbol 1}$ measurable. Assume that there exist constants $0<b_{i,k_i}, a_{i,k_i}$, $i\in[d], k_i\in I_i$, as well as a $c>0$, such that for each $\eta\subset[d]$:
\begin{equation}\label{ineq:lem3Condtition}
    \norm{\EE^\eta_\bk Z_{\bk}}_m\le c \cdot b_{\eta,\bk} \cdot a_{\eta^c,\bk}\,,
\end{equation}
where $b_{\eta,\bk} := \prod_{i\in\eta} b_{i,k_i}$ and $a_{\eta^c,\bk} = \prod_{i\in\eta^c}a_{i,k_i}$, as before. Then, it follows that for each $\theta,\eta\subset[d]$:
\begin{equation}\label{ineq:lem3Result}
    \norm{\EE^\eta_{\mathbf y}\sum_{\bk\in I_{\theta}} Z_{\bk}}_m\lesssim c\sum_{\theta\setminus\eta = \theta_1\dot\cup \theta_2} \left(\sum_{\bk\in I_{\theta\setminus\theta_2}} b_{\eta\cup\theta_1,\bk}\right) \left(\sum_{\bl\in I_{\theta_2}} a_{(\eta\cup\theta_1)^c,\bl}^2\right)^{\frac 12}\,,
\end{equation}
where we some over all disjoint $\theta_1,\theta_2$, such that $\theta\setminus\eta=\theta_1\cup\theta_2$. It especially follows, that for $\eta =\emptyset$:
\begin{equation*}
    \norm   {\sum_{\bk\in I_{\theta}} Z_{\bk}}_m\lesssim c\cdot a_{\theta^c,\by}\sum_{\theta = \theta_1\dot\cup \theta_2} \left(\sum_{\bk\in I_{\theta_1}} b_{\theta_1,\bk}\right) \left(\sum_{\bl\in I_{\theta_2}} a_{\theta_2,\bl}^2\right)^{\frac 12}\,.
\end{equation*}
\end{lem}
\begin{rem}
    The formula \eqref{ineq:lem3Result} becomes rather complicated, since we allow $\theta,\eta\subset[d]$ to be arbitrary. In particular, we want to be able to analyze $\EE^i_{y_i}\sum_{\bk\in I_{\theta}} Z_{\bk}$ for $i\notin\theta$, so we do not want to postulate $\eta\subset\theta$. While this complicates the formula, \eqref{ineq:lem3Result} is precisely what one would expect: For all $i\in\eta$, we have a conditional expectation $\EE^i_{y_i}$, so the right-hand-side contains the factor $\sum_{k_i\in I_i} b_{i,k_i}$ if $i\in\theta$, and $b_{i,y_i}$ if $i\notin\theta$. If an index $i$ is neither in $\theta$ nor in $\eta$, the right-hand-side gets the factor $a_{i,y_i}$. Last but not least, for an index $i\in\theta\setminus\eta$, the one-dimensional BDG-inequality gives us the factor $\sum_{k_i\in I_i} b_{i,k_i} + \left(\sum_{k_i\in I_i} a^2_{i,k_i}\right)^\frac 12$. If one multiplies out all these terms, one gets precisely \eqref{ineq:lem3Result}.
\end{rem}
\begin{proof}
    We prove the claim by induction over $\abs{\theta\setminus\eta}$. For $\eta = \theta$, we see that
    \begin{equation*}
        \norm{\EE^\theta_{\mathbf y} \sum_{\bk\in I_\theta} Z_\bk}_m \le \sum_{\bk\in I_\theta}\norm{\EE^\theta_{\bk} Z_\bk}_m\le c\cdot a_{\theta^c,\by}\sum_{\bk\in I_\theta} b_{\theta,\bk}
    \end{equation*}

    \noindent holds. Assume the claim is already proven for index sets $\theta, \eta$ with $\abs{\theta\setminus\eta}\le N$. Then for a $\abs{\theta\setminus\eta} = N+1$, w.l.o.g let $1\in\theta\setminus\eta$. We apply the one-dimensional BDG inequality \eqref{ineq:BDG1d} to the sum $\norm{\sum_{k_1\in I_1}\left(\EE^\eta_\by\sum_{\bx\in I_{\theta\setminus\{1\}}} Z_{(k_1,\bx)}\right)}_m$, where we use the notation $(k_1,\bx) := (k_1,x_2,\dots,x_d)$:
    \begin{align*}
        \norm{\EE^\eta_{\mathbf y} \sum_{\bk\in I} Z_\bk}_m&\lesssim \sum_{k_1\in I_1}\norm{\EE_{(k_1,\mathbf y)}^{\eta\cup\{1\}}\sum_{\bk\in I_{\theta\setminus\{1\}}} Z_{(k_1,\bk)}}_m+ \left(\sum_{k_1\in I_1} \norm{\EE_{\mathbf y}^{\eta}\sum_{\bk\in I_{\theta\setminus\{1\}}} Z_{(k_1,\bk)}}_m^2\right)^{\frac 12}\\
        &\lesssim c\sum_{k_1\in I_1}\sum_{(\theta\setminus\{1\})\setminus\eta = \theta_1\dot\cup\theta_2} \left(\sum_{\bk\in I_{(\theta\setminus\{1\})\setminus\theta_2}} b_{\eta\cup\{1\}\cup\theta_1,(k_1,\bk)}\right)\left(\sum_{\bl\in I_{\theta_2}} a_{(\eta\cup\{1\}\cup\theta_1)^c,(k_1,\bl)}^2\right)^{\frac 12} \\
        &\qquad + c\left(\rule{0cm}{1.3cm}\right. \sum_{k_1\in I_1}\underbrace{\left[\sum_{(\theta\setminus\{1\})\setminus\eta = \theta_1\dot\cup\theta_2} \left(\sum_{\bk\in I_{(\theta\setminus\{1\})\setminus\theta_2}} b_{\eta\cup\theta_1,(k_1,\bk)}\right)\left(\sum_{\bl\in I_{\theta_2}} a_{(\eta\cup\theta_1)^c,(k_1,\bl)}^2\right)^{\frac 12}\right]^2}_{U_{k_1}^2}\left.\rule{0cm}{1.3cm}\right)^{\frac 12}\,,
    \end{align*}
    note that the second term $U_{k_1}$ is of the form $U_{k_1} = a_{1,k_1}\cdot R$ for an $R$ independent of $k_1$. Thus, we can write $\sum_{k_1} U_{k_1}^2 = R^2\cdot\sum_{k_1}a_{1,k_1}^2$ to get
    \begin{align*}
        \norm{\EE^\eta_{\mathbf y} \sum_{\bk\in I} Z_\bk}_m&\lesssim c\sum_{(\theta\setminus\{1\})\setminus\eta = \theta_1\dot\cup\theta_2} \left(\sum_{\bk\in I_{\theta\setminus\theta_2}} b_{\eta\cup(\theta_1\cup\{1\}),\bk}\right)\left(\sum_{\bl\in I_{\theta_2}} a_{(\eta\cup(\theta_1\cup\{1\}))^c,\bl}^2\right)^{\frac 12} \\
        &\qquad + c\sum_{(\theta\setminus\{1\})\setminus\eta = \theta_1\dot\cup\theta_2} \left(\sum_{\bk\in I_{\theta\setminus(\theta_2\cup\{1\})}} b_{\eta\cup\theta_1,\bk}\right)\left(\sum_{\bl\in I_{\theta_2\cup\{1\}}} a_{(\eta\cup\theta_1)^c,\bl}^2\right)^{\frac 12}\\
        &= c\sum_{\theta\setminus\eta=\theta_1\dot\cup\theta_2} \left(\sum_{\bk\in I_{\theta\setminus\theta_2}}b_{\eta\cup\theta_1,\bk}\right)\left(\sum_{\bl\in I_{\theta_2}} a_{(\eta\cup\theta_1)^c,\bl}^2\right)^{\frac 12}\,.
    \end{align*}    
    This shows the lemma.
\end{proof}
\noindent Recall that the space $C_2^{\alpha,\beta,\gamma}L_m$ is defined as the space of $2d$ parameter processes $(\Xi_{\bs,\bt})_{\bs,\bt\in\Delta_\bT}$, such that for all $(\bs,\bt)\in\Delta_\bT$ we have:

\begin{itemize}
    \item $\Xi_{\bs,\bt}\in L_m$.
    \item $\Xi_{\bs,\bt}$ is $\cF_\bt$-measurable for some commuting filtration $(\cF_\bt)_{\bt\in[\boo,\bT]}$.
    \item We have 
    \begin{align*}
        \norm{\Xi_{\bs,\bt}}_m &\lesssim \prod_{i\in[d]}\abs{t_i-s_i}^\alpha = m_{\alpha[d]}(\bs,\bt) \\
        \norm{\EE^\eta_\bs\delta_\bu^\theta\Xi_{\bs,\bt}}&\lesssim \prod_{i\in\theta^c}\abs{t_i-s_i}^\alpha\prod_{i\in\theta\setminus\eta}\abs{t_i-s_i}^{\beta}\prod_{i\in\eta}\abs{t_i-s_i}^{\beta+\gamma} = m_{\alpha\theta^c+\beta\theta+\gamma\eta}(\bs,\bt)\,,
    \end{align*}

    \noindent for all $\bs\le\bu\le\bt$ as well as $\eta\subset\theta\subset[d]$ with $\theta\neq\emptyset$.
\end{itemize}

\noindent While highly technical, it is now straight-forward to generalize Lemma \ref{lem:stochastic sewing simplified} to the following lemma:

\begin{lem}[Stochastic multiparameter sewing]\label{lem:StochSewing}
Let $\Xi\in C_2^{\alpha,\beta,\gamma} L_m$. Assume $m\ge 2$, $\alpha\in(0,1),\gamma\in(0,\frac 12), \beta\in(\frac 12,\infty)$ with $\beta+\gamma>1$. Then the following limits exist for any $\theta\subset[d]$ for each sequence of grid-like partitions $\cP_\bn$, such that the mesh-size $\abs{\cP_\bn}$ vanishes as $\bn\to\infty$:
\begin{equation*}
    \cI^{\theta}\Xi_{\bs,\bt} := L_m-\lim_{\bn\to\infty} \sum_{[\bu,\bv]\in \cP^\theta_\bn} \Xi_{\bu,\bv}\,.
\end{equation*}
Furthermore, these limits are the unique family of two-parameter processes $(\cI^\theta\Xi_{\bs,\bt})_{\theta\subset[d]}$ in $L_m(\Omega)$ fulfilling the following properties:
\begin{itemize}
    \item $\cI^\theta\Xi_{\bs,\bt}$ is $\cF_\bt$-measurable for all $\bs\le\bt$.
    \item $\cI^\emptyset\Xi = \Xi$.
    \item $\cI^{\theta}\Xi$ is additive in $[\boo,\bT_\theta]$, i.e. for all $\emptyset\neq\eta\subset\theta$ and $\bs\le\bu\le\bt$, $\delta^\eta_\bu \cI^\theta\Xi_{\bs,\bt} = 0$.
    \item There exists a constant $C>0$ such that the following inequalities hold for all $\eta\subset\theta\subset[d]$, $\bs\le\bt$:
    \begin{equation}\label{ineq:uniqueProperty}
        \norm{\EE^\eta_\bs\sum_{\theta'\subset\theta} (-1)^{\abs{\theta'}} \cI^{\theta'}\Xi_{\bs,\bt}}_m \le C\norm{\Xi}_{\alpha,\beta,\gamma,m}m_{\alpha\theta^c+\beta\theta+\gamma\eta}(\bs,\bt)\,.
    \end{equation}
\end{itemize}
\end{lem}
\noindent The proof is the same as for Lemma \ref{lem:stochastic sewing simplified} combined with Lemma \ref{lem:every partition converges}, if one replaces Lemma \ref{lem:muti-parameterBDG-simplified}'s BDG inequality with the one from Lemma \ref{multiparameter BDG}.

\section{Local non-determinism of Gaussian random fields}
\label{LND section}

\noindent Our main application of the just established multiparameter stochastic sewing lemma consists of regularity estimates of Gaussian stochastic fields. However, in order to obtain such results, we require processes that enjoy another important structural property: local non-determinism (or abbreviated LND in the following). In the one-parameter setting, it is by now quite well understood that LND is the crucial property allowing to establish pathwise regularization by noise results in the spirit of \cite{gubicat}, as it implies $\rho$-irregularity of the process (see \cite[Theorem 30]{prevalence}). It is therefore natural to expect that LND plays an equally instrumental role in the multiparameter setting we are concerned with. We therefore begin this section with an overview of known results in the literature on LND for stochastic fields due to Xiao and coauthors, before introducing the notions of additive and multiplicative LND that will be used later on. As an example, we briefly discuss the fractional Brownian sheet, for which the LND property can be easily checked by hand.

\subsection{Sectorial and Strong LND}
In this subsection, we recall the well-established notions of sectorial local non-determinism and strong local non-determinism. We discuss the fractional Brownian sheet in this context.

\begin{defn}[Sectorial LND] 
\label{sectorial dfn}\cite[Section 2.4, (C3)]{Xiao} Let $\zeta \in \RR^d_+$. A Gaussian random field $\{W_\bt\}_{\bt\in \RR^d}$ admits the $\zeta$-sectorial local non-determinism property if there exists for any $\epsilon\in (0, 1]$ a constant $c(\epsilon)>0$ such that for any $m\geq 1$ and $\bu, \bt^1, \dots \bt^m\in [\epsilon, \infty)^d$ we have a.s.
\begin{equation*}
    \mbox{Var}(W_\bu\ |\ W_{\bt^1}, \dots W_{\bt^m})\geq  c(\epsilon)\sum_{i=1}^d \min_{1\leq k\leq m}|\bt^k_i-\bu_i|^{2\zeta_i}. 
\end{equation*}
\end{defn}

\noindent Let $\cF_\bs^{\boldsymbol\epsilon} := \overline{\sigma(W_\br~\vert~ \boldsymbol\epsilon<\br<\bs)}$. Then the sectorial LND property does not just hold by conditioning on finitely many time points $\bt^1,\dots,\bt^m$, but also if we condition on $\cF_\bt^\bepsilon$, as the following corollary shows:


\begin{cor}\label{cor:sectorial_non_determinism}
    A continuous Gaussian process $W_\bt$, with filtration $\cF_\bt^{\boldsymbol\epsilon}$ as above, that admits the $\zeta$-sectorial local non-determinism property satisfies a.s. for $\bt\ge \bs\ge \boldsymbol{\epsilon}$
   \[
        \mbox{Var}(W_\bt| \mathcal{F}_\bs^{\boldsymbol\epsilon})\geq c(\boldsymbol\epsilon)\sum_{i=1}^d |t_i-s_i|^{2\zeta_i}
   \]
\end{cor}

\begin{proof}
    Let $\bs_n^\bk = \boldsymbol\epsilon + \bk\cdot\frac{\bs-\boldsymbol{\epsilon}}{2^n}$ (recall that $\bk\cdot\bx = (k_1x_1,\dots,k_dx_d)$ denotes point-wise multiplication) for $k= \{0,\dots,2^n\}^d$ be a dyadic partition of $[\boldsymbol\epsilon,\bs]$\,. Further let $\cG_n = \overline{\sigma(W_{s_n^\bk}~\vert~\bk\in\{0,\dots,2^n\}^d)}$. Since $W_\bt$ is a continuous process and point-wise limits preserve measurability, we have that
    \begin{equation*}
        \cF_\bs^{\boldsymbol\epsilon} = \cG_\infty := \overline{\sigma\left(\bigcup_{n\in\NN}\cG_n\right)}\,.
    \end{equation*}
    It holds that $W_\bt\in L_2(\Omega)$, so we can apply Lévy's upwards theorem to $W_t$ and $W_t^2$ to get that
    \begin{align*}
        E(W_\bt~\vert~\cF_\bs^{\boldsymbol\epsilon}) &= \lim_{n\to\infty}E(W_\bt~\vert~\cG_n) \\
        E(W_\bt^2~\vert~\cF_\bs^{\boldsymbol\epsilon}) &= \lim_{n\to\infty}E(W_\bt^2~\vert~\cG_n)
    \end{align*}
    as a.s. limits. Thus, we get that a.s.
    \begin{align*}
        \mbox{Var}(W_\bt~\vert~ \cF_\bs^{\boldsymbol\epsilon}) &= \lim_{n\to\infty}(E(W_\bt^2~\vert~\cG_n)-E(W_\bt~\vert~\cG_n)^2) \\
        &\ge c(\boldsymbol\epsilon)\sum_{i=1}^d\abs{t_i-s_i}^{2\zeta_i}\,,
    \end{align*}
    where we used that $\min_{\bk}\abs{t_i-(s_n^{\bk})_i}\ge\abs{t_i-s_i}$ holds.
\end{proof}


\begin{rem}
    Remark that in the very last step, we exploited the ordering $\bs\leq \bt$ introduced in the notation yielding the filtration $(\mathcal{F}^\epsilon_\bs)_\bs$ as defined above. This step would not have been possible if instead, we would have worked with the filtration $(\mathcal{F}^{\epsilon, *}_\bs)_\bs$ (refer again to Figure \ref{filtrations}).
\end{rem}
\begin{defn}[Strong LND] \cite[Section 2.4, (C3')]{Xiao} Let $\zeta\in \RR^d_+$. A Gaussian random field $\{W_\bt\}_{\bt\in \RR^d}$ with stationary increments admits the $\zeta$-strong local non-determinism property if there exists a constant $c>0$ such that for any $m\geq 1$ and $\bu, \bt^1, \dots \bt^m\in \RR^d$ we have
\begin{equation*}
    \mbox{Var}(W_\bu\ |\ W_{\bt^1}, \dots W_{\bt^m})\geq  c\min_{1\leq k\leq m}\left(\sum_{i=1}^d |\bt^k_i-\bu_i|^{\zeta_i}\right)^2. 
\end{equation*}
\end{defn}
 \begin{cor}
        Let $X$ be a continuous Gaussian process and $ \mathcal{F}_\bs:=\overline{\sigma(W_\br~\vert~ \boldsymbol\epsilon<\br<\bs)}$ be its associated strong past natural filtration. Assume $X$ admits the $\zeta$-strong local non-determinism property satisfies for $\bt>\bs$
        \begin{equation}
            \mbox{Var}(W_\bt| \mathcal{F}_\bs)\geq c\sum_{i=1}^d |t_i-s_i|^{2\zeta_i}
            \label{strong LND impli}
        \end{equation}
    \end{cor}
    \begin{proof}
        This follows from a simple adaption of the  proof in Corollary \ref{cor:sectorial_non_determinism}. 
    \end{proof}
We now introduce the main example we shall discuss throught this section and the remainder of the paper.
\begin{defn}[Fractional Brownian sheet]
\label{fbs}
    For a given vector $H=(H_1, \dots, H_d)$ with $H_i \in (0,1)$ we call a real valued centered Gaussian field with covariance function 
    \[
    \mathbb{E}[W_\bt^HW_\bs^H]=\prod_{i=1}^d \left( |s_i|^{2H_i}+|t_i|^{2H_i}-|t_i-s_i|^{2H_i}\right)
    \]
    a $H$-fractional Brownian sheet. For $W$ a Brownian sheet, the stochastic field
    \begin{equation}
        W^H_\bt=\kappa_H^{-1/2}\int_{-\infty}^{t_1}\dots \int_{-\infty}^{t_d} \prod_{i=1}^d g_{H_i}(s_i, t_i)dW_{\bs}
        \label{moving average}
    \end{equation}
   is an $H$-fractional Brownian sheet, where
    \[
    g_{H_i}(s, t)=(t-s)_+^{H_i-1/2}-(-s_i)_+^{H_i-1/2}, \qquad \kappa_H^2=\int_{-\infty}^{1}\dots \int_{-\infty}^{1} \prod_{i=1}^d g^2_{H_i}(s_i, 1)d\bs.
    \]
    The formula \eqref{moving average} is also called moving average representation of the fractional Brownian sheet. If $W^H_{j}$ with $j\leq n$ are independent copies of an $H$-fractional Brownian sheet, we call the vector valued stochastic field given by $W^H=(W^H_1, \dots, W^H_n)$ a $(d, n)$-fractional Brownian sheet. 
\end{defn}

\begin{example}\cite[Theorem 1]{Wu2007}
If $B^H$ is an  $H$-fractional Brownian sheet, then it admits the $H$-sectorial local non-determinism property.
\end{example}

\begin{rem}
    \label{with zero}
    Note that in Definition \ref{sectorial dfn}, the fact that we strictly bounded away from the origin is instrumental for fraction Brownian sheets to fall into this class. Indeed, if $W^H$ is an $H=(H_1, \dots H_d)$ fractional Brownian sheet, then for $\bt=(t_1, 0, \dots, 0)$ we have
    \[
    0=\prod_{i=1}^d |t_i|^{2H_i}=\mbox{Var}(W_\bt)=\mbox{Var}(W_\bt |\mathcal{F}_0)<\sum_{i=1}^d |t_i|^{\zeta_i}
    \]
    for any $\zeta\in \RR^d_+$. We can therefore in particular not expect strong local non-determinism to hold. Remark also that this argument can be applied to Riemann-Liouville fields of the form 
    \[
    X_\bt=\int_{0}^{t_1}\dots \int_{0}^{t_d}K(\bt, \bs)W(d\bs)
    \]
    where $K$ such that $\int_{0}^{t_1}\dots \int_{0}^{t_d}K^2 (\bt, \bs)d\bs <\infty$. Indeed, note that
    \[
    0=\int_{0}^{t_1}\int_0^0\dots \int_{0}^{0}K^2(\bt, \bs)d\bs =\mbox{Var}(X_\bt |\mathcal{F}_\bs)<\sum_{i=1}^d |t_i|^\zeta
    \]
    for any $\zeta\in \RR^d_+$, meaning such fields can not be expected to be strongly locally non-deterministic. 
\end{rem}

\subsection*{Towards multiplicative LND} Let us point out that the proof of \cite[Theorem 1]{Wu2007} is rather involved, as it exploits Fourier analytic arguments on the level of the harmonizable representation of $W$. In the following, let us provide some non-optimal yet insightful calculations for the fractional Brownian sheet which also motivate our later Definition \ref{def:multiplicative LND} and might be of independent interest. For the sake of transparency and readability, we focus on the case $d=2$ to make the arguments as explicit and instructive as possible. However, the arguments presented can be directly generalized. We will crucially make use of the moving average representation of the fractional Brownian sheet
\begin{equation}
    B^H(\bt)=\kappa_H^{-1}\int_{-\infty}^{t_1}\int_{-\infty}^{t_2}\underbrace{\prod_{j=1}^2 \left((t_j-s_j)_+^{H_j-1/2} -(-s_j)_+^{H_j-1/2}\right)}_{=:g(\bt, \bs)}W(d\bs),
\end{equation}
where $\kappa_H$ is some normalization constant. We now split up our integration domain into the disjoint domains (refer to Figure \ref{areas}) 
\begin{equation*}
    \begin{split}
        I_1=[\bs, \bt], \quad I_2=&[(s_1, 0), (t_1, s_2)], \quad I_3=[(0, s_2), (s_1, t_2)]\\
        R_1=[-\infty, \bs], \quad R_2=&[(s_1, -\infty), (t_1, 0)], \quad R_3=[(-\infty, s_2), (0, t_2)]
    \end{split}
\end{equation*}

\begin{figure}
    \centering
    \includegraphics[scale = .7]{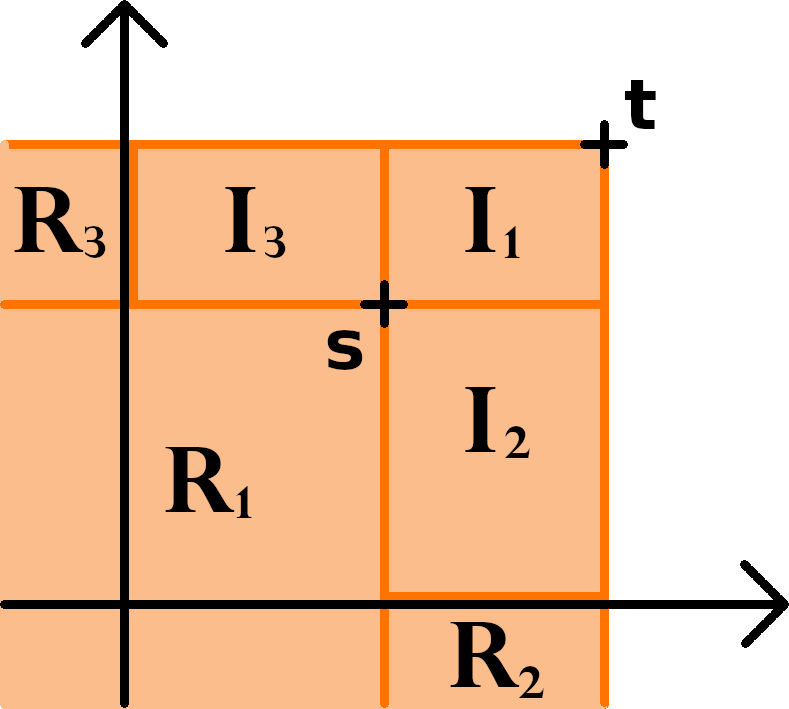}
    \caption{Areas of the respective terms}
    \label{areas}
\end{figure}

\noindent
obtaining readily 
\[
\kappa_H B^H(\bt)=\sum_{i=1}^3\int_{I_i}g(\bt, \bs)W(d\bs)+\sum_{i}^3\int_{R_i}g(\bt, \bs)W(d\bs)
\]
Note that as we integrate a Brownian sheet over disjoint domains, all the summands above are independent. Moreover, by definition of the filtration $\mathcal{F}$, we have that the term $\int_{R_1}g(\bt, \bs)W(d\bs)$ is $\mathcal{F}_\bs$ measurable, while all other summands are independent of $\mathcal{F}_\bs$. We can thus infer that 
\begin{align*}
    &\mathbb{E}[\left(\sum_{i=1}^3\int_{I_i}g(\bt, \bs)W(d\bs)+\sum_{i=1}^3\int_{R_i}g(\bt, \bs)W(d\bs)\right)|\mathcal{F}_\bs]\\
    =&\int_{R_1}g(\bt, \bs)W(d\bs)+\mathbb{E}[\sum_{i=1}^3\int_{I_i}g(\bt, \bs)W(d\bs)+\sum_{i=2}^3\int_{R_i}g(\bt, \bs)W(d\bs)]=\int_{R_1}g(\bt, \bs)W(d\bs)
\end{align*}
We, therefore, infer that 
\[
\kappa_H(B^H_\bt-\mathbb{E}[B^H_\bt|\mathcal{F}_\bs])=\sum_{i=1}^3\int_{I_i}g(\bt, \bs)W(d\bs)+\sum_{i=2}^3\int_{R_i}g(\bt, \bs)W(d\bs)
\]
Again because of pairwise independence of the above summands and independence from $\mathcal{F}_\bs$ we obtain 
\[
\kappa_H^2\mbox{Var}(B^H_\bt|\mathcal{F}_\bs)=\sum_{i=1}^3\int_{I_i}g^2(\bt, \bs)d\bs+\sum_{i=2}^3\int_{R_i}g^2(\bt, \bs)d\bs
\]
As all of the above summands are positive, we may neglect the integral coming from the domains $R_2, R_3$ for lower bounds (remark that indeed these domains can in any way not give rise to functions of increments $(t_2-s_2)$ or $(t_1-s_1)$ due to their structure and can thus not help in establishing local non-determinism properties). For the remaining terms, a direct calculation yields
\begin{equation}
    \mbox{Var}(B^H_\bt|\mathcal{F}_\bs)\gtrsim (t_1-s_1)^{2H_1}(t_2-s_2)^{2H_2}+(t_1-s_1)^{2H_1}(t_2^{2H_2}-(t_2-s_2)^{2H_2})+(t_2-s_2)^{2H_2}(t_1^{2H_1}-(t_1-s_1)^{2H_1})
    \label{explicit calc fbs}
\end{equation}
It is now readily checked that for $H_1, H_2\geq 1/2$, the above also implies 
\[
\mbox{Var}(B^H_\bt|\mathcal{F}_\bs)\gtrsim (t_1-s_1)^{2H_1}(t_2-s_2)^{2H_2}+(t_1-s_1)^{2H_1}s_2^{2H_2}+(t_2-s_2)^{2H_2}s_1^{2H_1}
\]
and thus, for $\bs\geq \epsilon$ 
\[
\mbox{Var}(B^H_\bt|\mathcal{F}_\bs)\gtrsim (\epsilon^{2H_2}+\epsilon^{2H_1})\left((t_1-s_1)^{2H_1}+(t_2-s_2)^{2H_2}\right)
\]
meaning we obtain the sectorial local non-determinism property above. In the case $H_i\leq 1/2$, it is possible to perform similar calculations, provided we restrict ourselves to a square $[\epsilon, \bT]$ and we allow the constant $c(\epsilon)$ in Definition \ref{sectorial dfn} of sectorial LND to depend on $\bT$ as well. Either way, the above calculations convey that due to the multiplicative structure of the fractional Brownian sheet (i.e. factorizing covariance kernel), 'additive' LND (formalized in Definition \ref{def:additive LND} below) can only be obtained at the price of considering fields away from zero. Also with respect to later applications to regularity studies of associated local times, this motivates the notion of 'multiplicative LND' introduced below.

\subsection{Additive and Multiplicative LND}
Let us point out that both sectorial LND and strong LND imply \eqref{strong LND impli} (albeit in the first case only 'away from zero'). As this property is instrumental in our study of associated local times, we isolate it in the following definition.

\begin{defn}[Additive LND]\label{def:additive LND}
    Let $\{X_\bt\}_{\bt\in \RR^d}$ be a Gaussian random field with $\mathcal{F}$ its natural strong past filtration. We say $X$ is additively local non-deterministic if there exists a constant $c>0$ and a multi-index $\zeta \in (0,1)^d$ such that for any $\bt\ge\bs\ge\mathbf{0}$ we have 
    \begin{equation}\label{eq:add LND}
        \mbox{Var}(X_{\bt}|\cF_\bs)\geq c \sum_{i=1}^d |t_i-s_i|^{2\zeta_i}. 
    \end{equation}
\end{defn}

\begin{defn}[multiplicative LND]\label{def:multiplicative LND}
         Let $\{X_\bt\}_{\bt\in \RR^d}$ be a Gaussian random field with $\mathcal{F}$ its natural strong past filtration. We say $X$ is multiplicatively local non-deterministic if there exists a constant $c>0$ and a multi-index $\zeta \in (0,1)^d$ such that for any $\bt\ge\bs\ge\mathbf{0}$ we have 
    \begin{equation*}
        \mbox{Var}(X_{\bt}|\cF_\bs)\geq c \prod_{i=1}^d |t_i-s_i|^{2\zeta_i}. 
    \end{equation*}
\end{defn}
\begin{example}
  From \eqref{explicit calc fbs}, it follows immediately that any $H$-fractional Brownian sheet is $H$-multiplicatively LND. By Remark \ref{with zero}, it is not $H$-additively LND.
\end{example}
\noindent In the following Lemma, we observe that more generally, additive LND is a property that can only come from boundary terms in the following sense. We restrict ourselves again to the case $d=2$ for readability, but mention that similar arguments can be easily extended to the general case.


\begin{lem}
\label{no additive LND}
    Let $(Z_\bt)_\bt$ be a two-parameter stochastic field with deterministic boundary, i.e. $Z_{(t_1, 0)}, Z_{(0, t_2)}$ deterministic for any $t_1, t_2\geq0$. Then $Z$ can not be additively LND.
\end{lem}
\begin{proof}
    We have the algebraic decomposition
    \[
    Z_\bt=Z_\mathbf{0}+\square^{(1)}_{\boo, (t_1, 0)}Z+\square^{(2)}_{\boo, (0, t_2)}Z+\square_{\boo, \bt}Z.
    \]
    Since the boundary terms are assumed to be deterministic, we have for $\bt=(t_1, 0)$ that $\square_{\boo,\bt} Z =0$, and therefore  
    \[
    \mbox{Var}(Z_\bt)=\mbox{Var}(\square_{\boo, \bt}Z)=\mbox{Var}(0)=0<|t_1|^\zeta
    \]
    for any $\zeta\in \RR^2_+$. The field $Z$ does therefore not satisfy  \eqref{eq:add LND} for any $\bt\ge\bs\ge \mathbf{0}$
\end{proof}
\subsection*{Additive LND from boundary terms} On the other hand, if $(Z_\bt)_\bt$ be a two parameter stochastic field such that  $Z_{(t_1, 0)}, Z_{(0, t_2)}, \square_{\boo,\bt}Z$ are independent for any $t_1, t_2>0$, then we have 
\[
\mbox{Var}(Z_\bt)\geq \mbox{Var}(\square^{(1)}_{\boo, (t_1, 0)}Z)+\mbox{Var}(\square^{(2)}_{\boo, (0, t_2)}Z)
\]
i.e. we can expect that additive LND might hold thanks to the effects of the boundary. To illustrate this on a concrete example, let us consider stochastic fields of the form
\begin{equation}
    X_{t_1, t_2} = \int_{-\infty}^{t_1}K_1(t_1, r)W^1(dr)+\int_{-\infty}^{t_2}K_2(t_1, r)W^2(dr) + \int_{-\infty}^{t_1}\int_{-\infty}^{t_2} K(\bt,\bs) \dd W_\bs.
\end{equation}
where $W^1, W^1$ are classical one-parameter Brownian motions and $W$ is Brownian sheet such that $Q^1, W^2, W$ are independent. We also assume for the kernels
\[
\int_{-\infty}^{t_1}K^2_1(t_1, r)dr+\int_{-\infty}^{t_2}K^2_2(t_2, r)dr+\int_{-\infty}^{t_1}\int_{-\infty}^{t_2}K^2(\bt, \bs)d\bs<\infty
\]
We consider the filtration $\mathcal{F}$ generated by $W^1, W^2, W$, which can be shown to be commuting. Going through similar calculations as above, one can show that 
\begin{equation}
    \mbox{Var}(X_\bt |\mathcal{F}_\bs)= \int_{s_1}^{t_1}K_1^2(t_1, r)dr+\int_{s_2}^{t_2}K_2^2(t_2, r)dr+\mbox{Var}\left(\int_{-\infty}^{t_1}\int_{-\infty}^{t_2}K(\bt, \bs)|\mathcal{F}_\bs\right)
\end{equation}
Now as discussed above, the last summand in the above expression will never be able to give us additive LND (at least under the integrability condition imposed on $K$). If we are interested in LND behavior on $[\boo, \bT]$ and not just away from zero, we can therefore employ the bound 
\begin{equation}\label{term:mult-part}
\mbox{Var}\left(\int_{-\infty}^{t_1}\int_{-\infty}^{t_2}K(\bt, \bs)\middle\vert\mathcal{F}_\bs\right)\geq \int_{s_1}^{t_1}\int_{s_2}^{t_2}K^2(\bt, \bs)d\bs 
\end{equation}
without losing critical terms. On the other hand, note that the boundary terms of $X$ do give rise to an additive structure. Provided for example $K_i^2(t, s)\geq |t-s|^{2H_i-1}$, and $K^2(\bt, \bs)\geq |t_1-s_1|^{2\tilde{H}_1-1}|t_2-s_2|^{2\tilde{H}_2-1}$ this yields 
\begin{equation*}
      \mbox{Var}(X_\bt |\mathcal{F}_\bs)\gtrsim |t_1-s_1|^{2H_1}+|t_2-s_2|^{2H_2}+|t_1-s_1|^{2\tilde{H}_1}|t_2-s_2|^{2\tilde{H}_2}
\end{equation*}
i.e. we see the multiplicative structure coming from the Riemann-Liouville type integral and the additive structure coming from the boundary terms. 

\begin{rem}
As is clear from the Definition \ref{sectorial dfn} and Corollary \ref{cor:sectorial_non_determinism}, for any stochastic field $X$ which is $\zeta$ sectorially LND with respect to its strong past natural filtration, the stochastic field 
\[
\tilde{X}_\bt:=X_{\bt+\mathbf{\bepsilon}}
\]
is $\zeta$-additively LND, where $\mathbf{\bepsilon}=(\epsilon, \dots \epsilon)$. Again, in this easy observation, the role of boundary for additive LND becomes clear: While the fractional Brownian sheet for example is zero on the boundary of $[0, \infty)^d$, we do not have additive LND right away. However, if we shift the field away from zero as above, the terms we cut away in \eqref{term:mult-part} become non-trivial, thus allowing for additive LND. 
\end{rem}

\section{Local time regularity of locally non-deterministic Gaussian fields}\label{sec_Local_Time}
\noindent Let $(W_\bt)_\bt$ be a $(d, n)$ Gaussian field. We will here establish the regularity of the associated occupation measure $\mu$ on a scale of Bessel potentials, provided some additive or multiplicative LND condition is imposed. The occupation measure is defined in the following way
\begin{equation*}
    \mu(A)=\lambda(\{\mathbf{0}\leq \bs \leq \bt| W_{\bs}\in A\}),  \quad A\subset \RR^n
\end{equation*}
where $\lambda$ denotes the Lebesgue measure. 
If the occupation measure is absolutely continuous with respect to the Lebesgue measure it admits a density called the local time (associated to $W$).

The Fourier transform of the occupation measure is given by the following
\begin{align*}
    \square_{\bs, \bt}\hat{\mu}(z)&=\int_{\bs}^{\bt} \exp(i\langle W_{\br}, z\rangle)d\br.
\end{align*}

\noindent
With the goal of proving joint space-time regularity of the occupation measure $\mu$ we will apply use the multiparameter stochastic sewing lemma to show that when defining  
\[
A_{\bs, \bt}:=\mathbb{E}[\int_{\bs}^{\bt} \exp(i\langle W_{\br}, z\rangle)d\br|\mathcal{F}_\bs], 
\]
then  $A$ admits a stochastic sewing $IA$,  and moreover, the sewing recovers our original quantity of interest in the sense that
\[
(IA)_\bt=\int_0^\bt \exp(i\langle W_\br, z\rangle)d\br. 
\]
To this aim, we will use the  simplified setting of the stochastic sewing lemma (Lemma \ref{lem:stochastic sewing simplified}). As explained just before Lemma \ref{lem:stochastic sewing simplified}, germs of the form
    \[
    A_{\bs,\bt}=\EE^\eta_{\bs}\left[ \int_\bs^{\bt} \exp(i\langle z,X_{\br}\rangle ) \dd \br \right]
    \]
  satisfy for all $\emptyset\neq\eta\subset \theta$
   \begin{equation*}
       \EE^\eta_{\bs}\left[ \delta^\theta_\bu A_{\bs,\bt}  \right]=0.
   \end{equation*}
due to the tower property, and so the simplified sewing lemma is directly applicable.

\begin{thm}\label{thm:bound moment fourier}
For some $\zeta\in (0,1)^d$ and $[\mathbf{0},\bT]\subset \RR_+^d$,  let $(W_\bt)_{\bt\in [\mathbf{0},\bT]}$ be an $n$-dimensional Gaussian random field on a filtered probability space $(\Omega,\cF,(\cF_{\bt}),\PP) $ which is $\zeta$--{\bf additive}--LND with respect to the strong past filtration $(\mathcal{F}_\bt)$ assumed to be commuting. 
Choose an $\eta\in \RR_+^d$ which is such that for all $i=1,...,d$
\begin{equation}
   \frac{1}{2}>\eta\cdot \zeta
\end{equation}
Then the Fourier transform of the occupation measure $\mu$ associated to $W$ satisfies for all $z\in \RR^n$
\begin{equation}
       \label{eq:additive stochastic bound}\|\square_{\bs,\bt}\hat{\mu} (z)\|_{m}\lesssim (1+|z|^2)^{-(\eta_1+...+\eta_d)} m_{\mathbf{1}-\eta\cdot \zeta}(\bs,\bt). 
\end{equation}
Furthermore, if $(W_\bt)_{\bt\in [\mathbf{0},\bt]}$ is $\zeta$--{\bf multiplicative}-LND with respect to the strong past filtration $(\mathcal{F}_\bt)$ assumed to be commuting, then for any $\theta>0$ satisfying 
\[
 \frac{1}{2 \max_i \zeta_i}>\theta
\]
for all $i=1,...,d$, 
then the following bound holds, 
\begin{equation}\label{eq:multiplicative stochastic bound}
\|\square_{\bs,\bt}\hat{\mu} (z)\|_{m}\lesssim (1+|z|^2)^{-\theta} m_{\mathbf{1}-\theta\zeta}(\bs,\bt). 
\end{equation}
\end{thm}

\begin{proof}
  We begin to recall the definition
    \begin{equation}
        A_{\bs,\bt}(z) := \int_{\bs}^\bt \EE[\exp(i\langle z,W_\br \rangle )|\cF_{\bs}]\dd r.
    \end{equation}
One can readily check that by the tower property of conditional expectations, we have that for any $\bs\le \bu\le \bt$  and $\emptyset\neq\eta\subset\theta\subset [d]$
\[
\EE^\eta_\bs\delta_\bu^\theta A_{\bs,\bt}(z)=0. 
\]
Furthermore, using the assumption that $W$ is a Gaussian random field, we have that 
\[
A_{\bs,\bt}(z) = \int_\bs^\bt \exp\left( i\langle \mu_\br^{\cF_\bs},z\rangle -\frac{1}{2}z^T \,\mathrm{Var}(W_\br|\cF_\bs) \, z\right)\dd \br, 
\]
where $\mu_\br^{\cF_\bs}=\EE[W_\br|\cF_\bs]$.
We, therefore, obtain the following bound 
\[
\|A_{\bs, \bt}(z)\|_m\leq \int_{\bs}^\bt \exp(-\frac{1}{2}z^T\mathrm{Var}(W_\br|\cF_\bs) z)d\br.
\]
Since $W_\br$ is a Gaussian random variable for each $\br$, the conditional variance $\mbox{Var}(W_\br|\mathcal{F}_\bs)$ is deterministic, see e.g. \cite{Bogachev2008GaussianM}. 
Using in  combination the assumption that $W$ is $\zeta$-additive LND we  find that 
\begin{equation}\label{eq:using Additive LND 1}
    \|A_{\bs, \bt}(z)\|_m\leq \int_{\bs}^\bt \prod_{i=1}^d \exp\left(-\frac{1}{2}|z|^2|r_i-s_i|^{2\zeta_i}\right)d\br, 
\end{equation}
and thus using that $e^{-ax^2}\leq e^a(a(1+x^2))^{-\eta/2}$ for any $a>0$ and  $\eta>0$, we have that 
\begin{equation}\label{eq: using additve LND cont.}
    \|A_{\bs, \bt}(z)\|_m\leq C e^{\sum_{i=1}^d T_i } \int_{\bs}^\bt \prod_{i=1}^d (1+|z|^2)^{-\eta_i/2}|r_i-s_i|^{-\zeta_i\eta_i} d\br, 
\end{equation}
and thus as long as $\eta_i\zeta_i< 1$, we have that 
\begin{equation}\label{eq:bound on A}
    \|A_{\bs, \bt}(z)\|_m\lesssim  C (1+|z|^2)^{-(\eta_1+...+\eta_d)/2} m_{\mathbf{1}-\eta\cdot \zeta}(\bs,\bt)
\end{equation}
where $\mathbf{1}-\eta\cdot \zeta=(1-\eta_1\zeta_1,...,1-\eta_d\zeta_d).$ We therefore conclude that $A(z)\in C^{\alpha}_2 L_m$, with $\alpha=\mathbf{1}-\eta\cdot \zeta$, and that $\norm{A(z)}_{\alpha,m} \lesssim (1+\abs{z}^2)^{-(\eta_1+\dots+\eta_d)/2}$ holds. Under the assumption that 
\begin{equation}
    1-\eta\cdot \zeta>{\frac{1}{2}},
\end{equation}
we can apply the special case of the stochastic sewing lemma in Corollary \ref{cor:sewing_alpha=beta} to conclude that there exists an additive function $\cA (z):[0,T]^d\rightarrow L_m$ such that 
\[
\cA_{\bs,\bt}(z)=\lim_{|\cP|\rightarrow 0} \sum_{[\bu,\bv]\in \cP } A_{\bu,\bv}(z). 
\]
In particular, using the bound in \eqref{ineq:uniqueness23}  we have that 
\begin{equation*}
    \|\cA_{\bs,\bt}(z)\|_{m}\lesssim (1+|z|^2)^{-(\eta_1+...+\eta_d)/2} m_{\mathbf{1}-\eta\cdot \zeta}(\bs,\bt)
\end{equation*}
For the bound in the case of $W$  being $\zeta$-multiplicative-LND, we must replace the additive LND condition applied in \eqref{eq: using additve LND cont.} with the multiplicative LND condition, and we would get 
\begin{equation}\label{eq:usingmultiplicative LND 1}
    \|A_{\bs, \bt}(z)\|_m\leq \int_{\bs}^\bt \exp\left(-\frac{1}{2}|z|^2\prod_{i=1}^d|r_i-s_i|^{2\zeta_i}\right)d\br.
\end{equation} 
Proceeding along the same lines as earlier, we see that for any $\theta>0$ satisfying $1/2\zeta_i>\theta$ for all $i=1,...,d$ we obtain the bound
\[
   \|A_{\bs, \bt}(z)\|_m\leq Ce^{\prod_{i=1}^d T_i} (1+|z|^2)^{-\theta/2}\int_{\bs}^\bt \prod_{i=1}^d |r_i-s_i|^{-\zeta_i\theta} d\br\simeq_{\mathbf{T},\zeta,\theta}  (1+|z|^2)^{-\theta/2}m_{\mathbf{1}-\zeta \theta}(\bs,\bt).  
\]
Note that in contrast to the case with additive LND, $\theta$ is now a positive number.

\noindent At last, we will prove that the stochastic sewing $\cA_{\bs,\bt}(z)$ satisfies
\[
\hat{\mu}_\bt (z)= \cA_{\boo,\bt}(z).
\]
To this end, let $\cP_n^{[d]}$ denote a grid like partition of $[\bs,\bt]$ consisting of $n$ intervals, which is such that $|\cP_n^{[d]}|\rightarrow 0 $ as $n\rightarrow \infty$. We begin to define the approximating sum
\[
\cP^{[d]}_n A(z)_{\bs,\bt}= \sum_{[\bu,\bv]\in \cP_n^{[d]}} A_{\bu,\bv}(z),  
\]
and recall that we have established the convergence $\cP_n^{[d]}A(z)_{\bs,\bt}\rightarrow \cA_{\bs,\bt}(z)$ in $L_m$ as $n\rightarrow \infty$. 
Observe that 
\begin{equation*}
    A_{\bs,\bt}(z) - \square_{\bs,\bt} \hat{\mu}(z) = \int_{\bs}^\bt \EE[\exp(i\langle z,W_\br \rangle )|\cF_{\bs}]-\exp(i\langle z,W_\br \rangle )\dd r, 
\end{equation*}
and by additivity of $\hat{\mu}$, we have 
\begin{equation}
\begin{aligned}
\|\cP^{[d]}_n A(z)_{\bs,\bt}- \square_{\bs,\bt} \hat{\mu}(z)\|_m \leq & \sum_{[\bu,\bv]\in \cP^n} \int_{\bu}^\bv \|\EE[\exp(i\langle z,W_\br \rangle )|\cF_{\bu}]-\exp(i\langle z,W_\br \rangle )\|_m \dd \br
\\
\leq &\left( \max_{[\bu,\bv]\in \cP^n} \max_{\br \in [\bu,\bv]} \|\EE[\exp(i\langle z,W_\br \rangle )|\cF_{\bu}]-\exp(i\langle z,W_\br \rangle )\|_m\right) m_{\mathbf{1}}(\bs,\bt),
\end{aligned}
\end{equation}
Furthermore, by addition and subtraction of the term $\exp(i\langle z,W_\bu \rangle ) $, it is readily seen that 
\[
\|\EE[\exp(i\langle z,W_\br \rangle )|\cF_{\bu}]-\exp(i\langle z,W_\br \rangle )\|_m \leq 2 \|\exp(i\langle z,W_\br \rangle )-\exp(i\langle z,W_\bu \rangle )\|_m,
\]
since we have that 
\[
\|\EE[\exp(i\langle z,W_\br \rangle )-\exp(i\langle z,W_\bu \rangle )|\cF_{\bu}]\|_m \leq \|\exp(i\langle z,W_\br \rangle )-\exp(i\langle z,W_\bu \rangle )\|_m. 
\]
By continuity of $W$, it is then readily checked that 
\[
\max_{[\bu,\bv]\in \cP^n} \max_{\br \in [\bu,\bv]} \|\EE[\exp(i\langle z,W_\br \rangle )|\cF_{\bu}]-\exp(i\langle z,W_\br \rangle )\|_m\rightarrow 0 \quad as \quad n\rightarrow \infty. 
\]
Thus we conclude that $\cA(z) = \hat{\mu}(z)$, which finishes the proof. 

\end{proof}

\begin{rem}
The bound established for the Fourier transform of the occupation measure can be seen in light of the concept of $\rho$-irregularity, discussed by Galeati and Gubinelli in \cite{galeati2020noiseless,prevalence}. While this concept is established there for (1 parameter) stochastic processes, it seems to be readily extendable multiparameter stochastic fields, and thus results related to prevalence of such fields would be interesting to study from this point of view. We leave deeper investigations into this relation for future studies. 
\end{rem}

With the above bound on the higher moments of the Fourier transform of the occupation measure, we are now in a position to prove Bessel potential regularity (for the definition, see \eqref{bessel_potential}) of the occupation measure associated to a continuous Gaussian field which is $\zeta$-LND.

\begin{thm}
\label{additive LND local time}
Let $\zeta\in (0,1)^d$. Assume $W$ is a $(d, n)$ Gaussian random field which is $\zeta$--{\bf additive}--LND with respect to its strong past natural filtration assumed to be commuting. Then its associated local time $L:[0,T]\times \RR^n \rightarrow \RR_+$ exists. Moreover, there exists a $\gamma\in (\frac{1}{2},1)^d$  such that for any $\alpha$ satisfying the bound 
\begin{equation}
\label{additive alpha}
\alpha<\sum_{i=1}^d \frac{1}{2\zeta_i}-\frac{n}{2}, 
\end{equation}
the local time satisfies $L\in C^\gamma([\mathbf{0},\bT];H^\alpha)$--$\PP$-a.s..

Furthermore, if $W$ is instead is  $\zeta$--{\bf multiplicative}--LND with respect to its strong past natural filtration assumed to be commuting, then there exists a $\gamma\in (\frac{1}{2},1)^d$ such that for any $\alpha>0$ satisfying
\begin{equation}
    \alpha< \frac{1}{2\max_i \zeta_i  }-\frac{n}{2},
    \label{multi alpha}
\end{equation}
the local time satisfies $L\in C^\gamma([\mathbf{0},\bT];H^\alpha)$--$\PP$-a.s..
\end{thm}

\begin{proof}
We will here prove that there exists a set $\Omega'\subset \Omega$ of full measure which is such that for all $\omega\in \Omega'$ then  for $\alpha $ satisfying \eqref{additive alpha} respectively \ref{multi alpha} and $\gamma>1/2$  we have
\begin{equation}\label{eq:pathwise bound}
   \|\square_{\bs,\bt}L(\omega) \|_{{H}^\alpha}^2 =\int_{\RR^d} (1+|z|^2)^{\alpha} |\square_{\bs,\bt}\hat{L}(z)|^2 \dd z \lesssim m_{\gamma}(\bs,\bt).  
\end{equation}
To this end, using the stochastic bounds from Theorem \ref{thm:bound moment fourier}, specifically \eqref{eq:additive stochastic bound} for the additive setting, using Minkowski's inequality, we see that for any $m\geq 2$
\[
\|\|\square_{\bs,\bt}L \|_{{H}^\alpha}\|_m\leq \left( \int_{\RR^d} (1+|z|^2)^{\alpha} \|\square_{\bs,\bt}\hat{L}(z)\|_m^2 \dd z \right)^{\frac{1}{2}} \lesssim m_{\mathbf{1}-\eta\cdot \zeta}(\bs,\bt) \left(\int_{\RR^d} (1+|z|^2)^{\alpha-(\eta_1+...+\eta_d)} \dd z \right)^\frac{1}{2}. 
\]
Since $\eta_i<1/2\zeta_i$ then in order for the integral term on the RHS above to be finite,  we get the following bound on $\alpha$
\[
\alpha<\sum_{i=1}^d \frac{1}{2\zeta_i}-\frac{n}{2}.
\]
It now follows by Kolomogorov's continuity theorem that there exists a set $\Omega'\subset\Omega$ of full measure such that for all $\omega\in \Omega'$ \eqref{eq:pathwise bound} holds. We conclude that $L\in C^\gamma([\mathbf{0},\bT];H^\alpha (\RR^d))$ -- $\PP$-a.s..

Going through essentially the same argument as in the above, one can also obtain an analogous result for $\zeta$--{\bf multiplicative}--LND Gaussian random fields, simply replacing the bound used from \eqref{eq:additive stochastic bound} with the subsequent bound \eqref{eq:multiplicative stochastic bound}.  
\end{proof}

\begin{rem}
    Using the bounds obtained in Theorem \ref{thm:bound moment fourier} one can surely prove joint space-time regularity of the occupation measure in more exotic function spaces, such as Besov spaces, or similar using the same strategy as above. We choose to showcase Sobolev spaces here due to their wide applications and brevity of the proof. 
\end{rem}

\section{Applications to regularization by noise for SDEs in the plain}
\label{regularizatio by noise section}
\noindent In light of the above results on local times of stochastic fields, we are able to immediately deduce associated regularization by noise results following \cite{Bechtold2023}. We briefly recall the setting investigated there before formulating new regularization by noise results as corollaries.

\bigskip

For a given continuous path $w:[0,T]^2\to \RR^n$ and a nonlinearity $b:\RR^n\to \RR^n$ consider the integral equation
\[
x_\bt=\xi_\bt+\int_0^{\bt} b(x_\bs)\dd \bs+w_\bt
\]
where $\xi_\bt=\xi^1_{t_1}\mathbf{1}_{t_2=0}+\xi^2_{t_2}\mathbf{1}_{t_1=0}$ is a given boundary condition. Using the transformation $y=x-w$, we obtain 
\begin{equation}
    y_\bt=\xi_\bt+\int_0^{\bt} b(y_\bs+w_\bs)\dd\bs
    \label{regularized eqn}
\end{equation}
Considering a highly fluctuating field $w$, it becomes clear that at least for smooth non-linearities $b$, the oscillations of $w$ dominate the oscillations of $y$. Hence, on small squares $[\bs, \bt]$, we have by the occupation times formula the local approximation 
\[
\int_\bs^\bt b(y_\br+w_\br)\dd\br \simeq \int_\bs^\bt b(y_\bs+w_\br)\dd\br=\int_{\RR^n} b(y_\bs-z)\square_{\bs, \bt}L^{-w}(z)\dd z=(b*\square_{\bs, \bt}L^{-w})(y_\bs)
\]
Due to the obtained spatial regularity of the local time, we can observe at this level a local gain of regularity. In particular, even if $b$ is only defined as a distribution, the function $z\to (b*\square_{\bs, \bt}L)(z)$ might be Lipschitz, provided $L$ is sufficiently smooth. Using sufficient gain in local regularity, one can obtain the so-called 2D nonlinear Young integral by summing up the above infinitesimal approximations along $[\boo, \bt]$, i.e.
\[
\lim_{|\mathcal{P}^n([\boo, \bt])|\to 0}\sum_{[\bu, \bv]\in \mathcal{P}^n([0, \bt])}(b*\square_{\bu, \bv}L^{-w})(y_\bu)=:\int_0^\bt b(y_\bs+w_\bs)\dd\bs
\]
We refer the reader to \cite[Section 3]{Bechtold2023} for the detailed construction. This can be shown to provide a consistent generalization of the right-hand side integral in \eqref{regularized eqn} and thus provides us with an object suitable for fixed-point arguments. Overall, one obtains the following result
\begin{thm}[\cite{Bechtold2023} Theorem 28]
  Assume $b\in H^\zeta(\RR^n)$ for $\zeta \in \RR$ and that $w\in C([0,T]^2;\RR^n)$ has an associated local time $L^{-w}\in C^\gamma_\bt H^\alpha(\RR^n)$ for some $\gamma\in (\frac{1}{2},1]^2$ and $\alpha \in \RR_+$. If  there exists an   $\eta\in (0,1)$  such that 
\begin{equation}
   \gamma(1+\eta)>1\quad  and \quad \zeta+\alpha>2+\eta, 
   \label{regularity condition}
\end{equation}
then there exists a unique solution $y\in C([0,T]^2;\RR^n)$ to equation \ref{regularized eqn}, where the right-hand side is understood as a 2D nonlinear Young integral. 
\end{thm}

\noindent
Combining our quantified regularity results for the local time of realizations of different stochastic fields (Theorem \ref{additive LND local time}) with the regularity condition \eqref{regularity condition}, we can therefore immediately deduce quantified regularization by noise results for SDEs driven by such stochastic fields.

\begin{cor}
\label{regularizing cor}
   Let $b\in H^{\rho}(\RR^n)$ for $\rho\in \RR$. Let $W$ be $(2, n)$ Gaussian stochastic field which is $(\zeta_1, \zeta_2)$-additive LND with respect to its strong past natural filtration that is assumed to be commuting. Suppose that 
    \[
   \rho+\frac{1}{2\zeta_1}+\frac{1}{2\zeta_2}-\frac{n}{2}>3
    \]
   then the problem \eqref{regularized eqn} admits a unique solution $y\in C([0, T]^2, \RR^n)$, where the right-hand side is understood as a 2D nonlinear Young integral. \\
  Let $W$ be $(2, n)$ Gaussian stochastic field which is $(\zeta_1, \zeta_2)$-multiplicatively LND with respect to its strong past natural filtration that is assumed to be commuting. Suppose that 
    \[
   \rho+\frac{1}{2\max(\zeta_1, \zeta_2)}-\frac{n}{2}>3
    \]
   then the problem \eqref{regularized eqn} admits a unique solution $y\in C([0, T]^2, \RR^n)$, where the right-hand side is understood as a 2D nonlinear Young integral. 
\end{cor}

\begin{rem}
    For the fractional Brownian sheet, remark that from the above Corollary \ref{regularizing cor} we obtain the same regularization by noise result as in \cite[Theorem 33]{Bechtold2023}, whose proof was based on the one parameter stochastic sewing lemma and self-similarity properties. It might seem potentially surprising at first sight that these one parameter considerations don't yield worse regularization results than the ones obtained with our multiparameter stochastic sewing lemma. However, this is explained by the lack of the additive LND property of the fractional Brownian sheet (Remark \ref{with zero}). More generally, in light of Lemma \ref{no additive LND}, improved regularization results with respect to one-parameter arguments as in \cite[Theorem 31]{Bechtold2023} should be expected only due to stochastic boundary terms.
    
\end{rem}

\section{Further perspectives and concluding remarks}
 In the present work, we addressed a first open question in \cite{Bechtold2023} concerning the refined space-time regularity estimates of local times for Gaussian fields, allowing to establish systematic regularization by noise results. The main tool we established towards this end is a multiparameter stochastic lemma, which is also of independent interest. We observed that stochastic fields with multiplicative covariance structure can't be additively LND and  that the local time regularity of $\zeta$-multiplicative LND fields appears to be dictated by the $\max_{i\leq d}\zeta_i$. On the other hand, additive LND appears to be a property that comes from boundary terms of the field suggesting that these are also responsible for the considerably higher regularity of associated local times (see Theorem \ref{additive LND local time}). \\
\\
Let us mention beyond the fractional Brownian field and Riemann-Liouville type stochastic fields, one could be also interested in the regularizing property of strongly LND stochastic fields.  Concrete examples of fields which admit the strong LND property include Gaussian processes whose spectral density satisfies
\[
    f(\lambda)\simeq \frac{1}{\left(\sum_{i=1}^d|\lambda_i|^{\zeta_i}\right)^{2+Q}}, \qquad \lambda\in \RR^d\backslash \{0\}.
    \]
(Refer to \cite[Theorem 3.2]{Xiao}). However, while for the fractional Brownian sheet, the commuting filtrations property is a direct consequence of the moving average representation, it appears unclear how this can be established for Gaussian fields with the above spectral density. Note in particular that the above growth condition rules out a factorization of the spectral measure suggesting that we are not dealing with a field of multiplicative covariance structure. This is consistent with our previous observation that one requires a non-multiplicative covariance structure to obtain additive LND. One potential future problem could thus consist in establishing the commuting filtrations property for Gaussian fields with the above spectral density. \\
\\
Another possible extension concerns the study of $\alpha$-stable stochastic fields, their local times, and regularizing effects. In the one parameter setting, this has been achieved in \cite{ling}. Given that strong local non-determinism for $\alpha$-stable fields has already been investigated in \cite{Xiao2011}, it should be possible to extend the results of Section \ref{sec_Local_Time} and \ref{regularizatio by noise section} to such fields using our multiparameter stochastic sewing lemma. \\
\\
Besides the approach to pathwise regularization by noise going through local time estimates, another way to establish regularization phenomena is by studying the averaging operator 
\[
T^W_\bt b(x)=\int_0^\bt b(x+W_\br)d\br 
\]
directly. In the one-parameter setting, this approach is for example taken in \cite{galeati2020noiseless}, where the authors are able to obtain sharper regularization by noise results (although at the price of having to deal with null-sets depending on the function $b$). Given an additively/ multiplicatively LND Gaussian field and with the multiparameter stochastic sewing lemma at hand, it should be possible to obtain similar regularity estimates for the averaging operator in the multiparameter setting. This should then lead to improved regularization by noise results for SDEs in the plain.

\bibliographystyle{alpha}
\bibliography{biblio}
\end{document}